\documentclass[sigconf]{acmart}
\pdfoutput=1

\newtheorem{thm}{Theorem}

\newtheorem{assump}{Assumption}

\newtheorem{lemma}{Lemma}
\newtheorem{definition}{Definition}

\newtheorem{rmq}{Remark}

\newtheorem*{assump*}{Assumption}

\newtheorem*{example*}{Example}
\newtheorem*{prop*}{Proposition}
\newtheorem*{thm*}{Theorem}
\newtheorem*{prv*}{Proof}

\usepackage{placeins}
\AtBeginDocument{%
  \providecommand\BibTeX{{%
    \normalfont B\kern-0.5em{\scshape i\kern-0.25em b}\kern-0.8em\TeX}}}

\setcopyright{acmcopyright}
\copyrightyear{2018}
\acmYear{2018}
\acmDOI{10.1145/1122445.1122456}

\acmConference[FOGA 2021]{XX}{XX}{XX}
\acmBooktitle{XX}
\acmPrice{15.00}
\acmISBN{978-1-4503-XXXX-X/18/06}

\newcommand{\newotc}[1]{{\textcolor{black}{#1}}}

\acmSubmissionID{123-A56-BU3}

\copyrightyear{2021}
\acmYear{2021}
\setcopyright{acmcopyright}\acmConference[FOGA '21]{Foundations of Genetic
Algorithms XVI}{September 6--8, 2021}{Virtual Event, Austria}
\acmBooktitle{Foundations of Genetic Algorithms XVI (FOGA '21), September 6--8,
2021, Virtual Event, Austria}
\acmPrice{15.00}
\acmDOI{10.1145/3450218.3477302}
\acmISBN{978-1-4503-8352-3/21/09}
\begin{document}

\title{Asymptotic convergence rates for averaging strategies}

\author{Laurent Meunier}
\email{laurentmeunier@fb.com}
\affiliation{%
  \institution{Paris-Dauphine University\\Facebook AI Research}
  \country{France}}
  
\author{Iskander Legheraba}
\email{iskander.legheraba@dauphine.psl.eu}
\affiliation{%
  \institution{Paris-Dauphine University}
    \country{France}}
\author{Yann Chevaleyre}
\email{yann.legheraba@dauphine.psl.eu}
\affiliation{%
  \institution{Paris-Dauphine University}
  \country{France}}
\author{Olivier Teytaud}
\email{oteytaud@fb.com}
\affiliation{%
  \institution{Facebook AI Research}
 \country{France}}

\renewcommand{\shortauthors}{L. Meunier, I. Legheraba, Y. Chevaleyre, O. Teytaud}

\begin{abstract}
      Parallel black box optimization consists in estimating the optimum of a function using $\lambda$ parallel evaluations of $f$. Averaging the $\mu$ best individuals among the $\lambda$ evaluations is known to provide better estimates of the optimum of a function than just picking up the best. In continuous domains, this averaging is typically just based on (possibly weighted) arithmetic means. Previous theoretical results were based on quadratic objective functions. In this paper, we extend the results to a wide class of functions, containing three times continuously differentiable functions with unique optimum. We prove formal rate of convergences and show they are indeed better than pure random search asymptotically in $\lambda$. We validate our theoretical findings with experiments on some standard black box functions.
\end{abstract}




\maketitle

\section{Introduction}
Finding the minimum of a function from a set of $\lambda$ points $(x_i)_{i\leq \lambda}$ and their images $(f(x_i))_{i\leq \lambda}$ is a standard task used for instance in hyper-parameter tuning \cite{bergstra2012random}, or control problems. While random search estimate of the optimum consists in returning $\arg\min f(x_i)_{i\leq \lambda}$, in this paper we focus on the similar strategy that consists in averaging the $\mu$ best samples, i.e. returning $\frac1\mu\sum_{i=1}^\mu x_{(i)}$ where $f(x_{(1)})\leq\ldots\leq f(x_{(\lambda)})$.

These kinds of strategies are used in many evolutionary algorithms such as CMA-ES. Although experiments show that these methods perform well, it is not still understood why taking the average of best points actually leads to a lower regret. In \cite{ppsnkbest}, it is proved in the case of quadratic functions that the regret is indeed lower for the averaging strategy than for pure random search. In this paper, we extend the result of this paper by proving convergence rates for a wide class of functions including three times continuously differentiable functions with unique optima.

\subsection{Related work}

\subsubsection{Better than picking up the best}
Given a finite number of samples $\lambda$ equipped with their fitness values, we can simply pick up the best, or average the ``best ones''~\cite{beyerbenefitofsex,ppsnkbest}, or apply a surrogate model~\cite{sm1,sm2,sm3,AST,sm4,bach}. Overall, the best is quite robust, but the surrogate or the averaging usually provides  better convergence rates. Using surrogate modeling is fast when the dimension is moderate and the objective function is smooth (simple regret in $O(\lambda^{-m/d})$ for $\lambda$ points in dimension $d$ with $m$ times differentiability, leading to superlinear rates in evolutionary computation~\cite{AST}). In this paper, we are interested in the rates obtained by averaging the best samples for a wide class of functions. We extend the results of~\cite{ppsnkbest} which only hold for the sphere function.

\subsubsection{Weighted averaging}
Among the various forms of averaging,
it has been proposed to take into account the fact that the sampling is not uniform (evolutionary algorithms in continuous domains typically use Gaussian sampling) in \cite{sumo}: we here simplify the analysis by considering a uniform sampling in a ball, though we acknowledge that this introduces the constraint that the optimum is indeed in the ball. \cite{arnoldweights,anneweights} have proposed weights depending on the fitness value, though they acknowledge a moderate impact: we here consider equal weights for the $\mu$ best.

\subsubsection{Choosing the selection rate}
The choice of the selection rate $\mu/\lambda$ is quite debated in evolutionary computation: one can find $\mu=\lambda/7$ \cite{amorales}, $\mu=\lambda/2$ \cite{cmsa}, $\mu=0.27\lambda$ \cite{escompr}, $\mu=\lambda/4$ \cite{HAN}, $\mu=\min(d,\lambda/4)$ \cite{chooselambda,fournierAlgorithmica} and still others in \cite{beyerbenefitofsex,jeb}. In this paper, we focus on the selection rate when the number of samples $\lambda$ is very large in the case of parallel optimization. In this case, the selection ratio would tend to $0$. We carefully analyze this ratio and derive convergence rates using this selection ratio.

\subsubsection{Taking into account many basins} While averaging the best samples, the non-uniqueness of an optimum might lead to averaging points coming from different basins. Thus we consider at first the case of a unique optimum and hence a unique basin. Then we aim to tackle the case where there are possibly different basins. Island models~\cite{islands} have also been proposed for taking into account different basins. \cite{ppsnkbest} has proposed a tool for adapting $\mu$ depending on the (non) quasi-convexity. In the present work, we extend the methodology proposed in \cite{ppsnkbest}.

\subsection{Outline}
In the present paper, we first introduce, in Section~\ref{sec:assumptions}, the large class of functions we will study, and study some useful properties of these functions in Section~\ref{sec:teclemmas}. Then, in Section~\ref{sec:randomsearch}, we prove upper and lower convergence rates for random search for these functions. In Section~\ref{sec:mubestrate}, we extend \cite{ppsnkbest} by showing that asymptotically in the number of samples $\lambda$, the handled functions satisfy a better convergence rate than random search. We then extend our results on wider classes of functions in Section~\ref{sec:wider}. Finally we validate experimentally our theoretical findings and compare with other parallel optimization methods.

\section{Beyond quadratic functions}
\label{sec:assumptions}
In the present section, we present the assumptions to extend the results from \cite{ppsnkbest} to the non-quadratic case. We will denote $B(0,r)$ the closed ball centered at $0$ of radius $r$ in $\mathbb{R}^d$ endowed with its canonical Euclidean norm denoted by $\lVert\cdot\rVert$. We will also denote by $\overset{\circ}{B}(0,r)$ the corresponding \emph{open} ball. All other balls intervening in what follows will also follow that notation. For any subset $S\subset B(0,r)$, we will denote $U(S)$ the uniform law on $S$. 

Let $f:B(0,r)\to \mathbb{R}$ be a continuous function for which we would like to find an optimum point $x^*$. The existence of such an optimum point is guaranteed by continuity on a compact set.
For the sake of simplicity, we assume that $f(x^{\star}) = 0$. We define the $h$-level sets of $f$ as follows.
\begin{definition}
Let $f:B(0,r)\to \mathbb{R}$ be a continuous function. The closed sublevel set of $f$ of level $h$ is defined as:
\begin{align*}
    S_{h}:=\{x\in B(0,r)\mid f(x)\leq h\}.
\end{align*}
\end{definition}
We now describe the assumptions we will make on the function $f$ that we optimize. 
\begin{assump}
\label{ass:principal}
$f:B(0,r)\to \mathbb{R}$ is a continuous function and admits a unique optimum point $x^\star$ such that $\lVert x^{\star}\rVert<r$. Moreover we assume that $f$ can be written:
\begin{align*}
f(x)=\left(x-x^\star\right)^T\mathbf{H}\left(x-x^\star\right)+ \left(\left(x-x^\star\right)^T\mathbf{H}\left(x-x^\star\right)\right)^{\alpha/2} \varepsilon(x-x^\star)
\end{align*}
for some bounded function $\varepsilon$ (there exists $M>0$ such that for all $x$, $\lvert \varepsilon(x)\rvert \leq M$), $\mathbf{H}$ a symmetric positive definite matrix and $\alpha>2$ a real number. 

\end{assump}
Note that $H$ is uniquely defined by the previous relation.
In the following we will denote by $e_1(\mathbf{H})$ and $e_d(\mathbf{H})$ respectively the smallest and the largest eigenvalue of $\mathbf{H}$. As $\mathbf{H}$ is positive definite, we have $0<e_1(\mathbf{H})\leq e_d(\mathbf{H})$. We will also set $\lVert x\rVert_{\mathbf{H}}=\sqrt{x^T\mathbf{H}x}$, which is a norm (\emph{the $\mathbf{H}$-norm}) on $\mathbb{R}^{d}$ as $\mathbf{H}$ is symmetric positive definite. We then have $f(x)=\lVert x-x^{\star}\rVert_{\mathbf{H}}^2+ \lVert x-x^{\star}\rVert_{\mathbf{H}}^{\alpha} \varepsilon(x-x^\star)$ 
\begin{rmq}[Why a unique optimum ?]The uniqueness of the optimum is an hypothesis required to avoid that chosen samples come from two or more wells for $f$. In this case the averaging strategy would lead to a mistaken point because points from the different wells would be averaged. \newotc{Nonetheless, multimodal functions can be tackled using our non-quasiconvexity trick (Section \ref{nonqc}).}\end{rmq}
	\begin{rmq}[Which functions $f$ satisfy Assumption~\ref{ass:principal}?] One may wonder if Assumption~\ref{ass:principal} is restrictive or not. We can remark that three times continuously differentiable functions satisfy the assumption with $\alpha=3$, as long as the \newotc{unique} optimum satisfies a strict second order stationary condition. \newotc{Also, we will see in Section \ref{invar} that results are immediately valid for strictly increasing transformations of any $f$ for which Assumption~\ref{ass:principal} holds, so that we indirectly include all piecewise linear functions as well as long as they have a unique optimum. } So the class of functions is very large, and in particular allows non symmetric functions to be treated, which might seem counter intuitive at first. 
\end{rmq}

The aim of this paper is to study a parallel optimization problem as follows. We sample $X_{1},\cdots,X_{\lambda}$ from the uniform distribution
on $B(0,r)$. Let $X_{(1)},\cdots,X_{(\lambda)}$ denote
the ordered random variables, where the order is given by the objective
function 
\[
f(X_{(1)})\leq\cdots\leq f(X_{(\lambda)}).
\]
We then introduce the $\mu$-best average 
\[
\overline{X}_{(\mu)}=\frac{1}{\mu}\sum_{i=1}^{\mu}X_{(i)}
\]

In the following of the paper, we will compare the standard random search algorithm (i.e. $\mu=1$) with the algorithm that consists in returning the average of the $\mu$ best points. To this end, we will study the expected simple regret for functions satisfying the assumption: \[\mathbb{E}\left[f(\overline{X}_{(\mu)})\right]\]

\section{Technical lemmas}
\label{sec:teclemmas}
In this section, we prove two technical lemmas on $f$ that will be useful to study the convergence of the algorithm. The first one shows that $f$ can be upper bounded and lower bounded by two spherical functions. 
\begin{lemma}
\label{lem:sandwich}
Under Assumption~\ref{ass:principal}, there exist two real numbers $0<l\leq L$, such that, for all $x\in B(0,r)$:
\begin{align}
\label{eq:lip-cond}
   l\lVert x-x^\star\rVert^2 \leq f(x)\leq L\lVert x-x^\star\rVert^2.
\end{align}
Moreover such $l$ and $L$ must satisfy  $0< l\leq e_1(\mathbf{H})\leq e_d(\mathbf{H})\leq L$.
\end{lemma}
\begin{proof}
As $\mathbf{H}$ is symmetric positive definite, we have the following
classical inequality for the $\mathbf{H}$-norm
\begin{equation}\label{eq:sym-mat-ineq}
e_{1}(\mathbf{H})\lVert x-x^{\star}\rVert^{2}\le\lVert x-x^{\star}\rVert_{\mathbf{H}}^{2}\le e_{d}(\mathbf{H})\lVert x-x^{\star}\rVert^{2}
\end{equation}
Now set for $x\in B(0,r)\setminus\{x^{\star}\}$
\[
\phi(x):=\frac{f(x)-f(x^{\star})}{\lVert x-x^{\star}\rVert^{2}}=\frac{\lVert x-x^{\star}\rVert_{\mathbf{H}}^{2}}{\lVert x-x^{\star}\rVert^{2}}(1+\lVert x-x^{\star}\rVert_{\mathbf{H}}^{\alpha-2}\varepsilon(x-x^{\star})).
\]
By the above inequalities, we have 
\[
e_{1}(\mathbf{H})^{(\alpha-2)/2}\lVert x-x^{\star}\rVert^{\alpha-2}\le\lVert x-x^{\star}\rVert_{\mathbf{H}}^{\alpha-2}\le e_{d}(\mathbf{H})^{(\alpha-2)/2}\lVert x-x^{\star}\rVert^{\alpha-2}.
\]
Thus, as $\alpha>2$, we obtain $\lVert x-x^{\star}\rVert_{\mathbf{H}}^{\alpha-2}\rightarrow_{x\rightarrow x^{\star}}0$.
By assumption, the function $\varepsilon$ is also bounded as $x\rightarrow x^{\star}$.
\\
We thus conclude that there exists $\delta>0$ such that, for all
$x\in\overset{\circ}{B}(x^{\star},\delta)$
\[
\frac{1}{2}e_{1}(\mathbf{H})\le\phi(x)\le2e_{d}(\mathbf{H}).
\]
Now notice that $B(0,r)\setminus\overset{\circ}{B}(x^{\star},\delta)$
is a closed subset of the compact set $B(0,r)$ hence it is also compact.
Moreover, by assumption $f$ is continuous on $B(0,r)$ and $f(x)>0=f(x^{\star})$
for all $x\neq x^{\star}.$ Hence $\phi$ is continuous and positive
on this compact set. Thus it attains its minimum and maximum on this
set and its minimum is positive. In particular, we can write, on this set, for
some $l_{0},L_{0}>0$
\[
l_{0}\le\phi(x)\le L_{0}.
\]
We now set $l=\min\{l_{0},\frac{1}{2}e_{1}(\mathbf{H})\}$. Note that
$l>0$ because $l_{0}>0$ and $e_{1}(\mathbf{H})>0$ (as $\mathbf{H}$
is positive definite). We also set $L=\max\{L_{0},2e_{1}(\mathbf{H})\}$
which is also positive. These are global bounds for $\phi$ which gives the first part of the result.\\
For the second part, let $\mathbf{u}_{1}$ be a normalized eigenvector
respectively associated to $e_{1}(\mathbf{H})$. Then 
\begin{align*}
\frac{f(x^{\star}+\epsilon\mathbf{u}_{1})}{\lVert\epsilon\mathbf{u}_{1}\rVert^{2}}=e_{1}(\mathbf{H})+\epsilon^{\alpha-2}\varepsilon(\epsilon\mathbf{u}_{1})
\end{align*}
Taking the limit as $\epsilon\to0$. we get that, if $l$ satisfies~\eqref{eq:lip-cond},
then $l\leq e_{1}(\mathbf{H})$. Similarly, we can prove that $L\geq e_{d}(\mathbf{H})$.\end{proof}
Secondly, we frame $S_h$ into two ellipsoids as $h\to 0$. This lemma is a consequence of the assumptions we make on $f$.
\begin{lemma}
\label{lemma:sandwich-set}
Under Assumption~\ref{ass:principal}, there exists $h_0\geq 0$ such that for $h\leq h_0$, we have $A_h\subset S_h\subset B_h$ where:
\begin{align*}
A_h:=\{x\mid  \lVert x-x^\star\rVert_{\mathbf{H}}\leq \phi_-(h)\}\\
B_h:=\{x\mid   \lVert x-x^\star\rVert_{\mathbf{H}}\leq \phi_+(h)\}
\end{align*}
with $\phi_-(h)$ and $\phi_+(h)$ two functions satisfying 
\begin{eqnarray*}\phi_-(h)&=&\sqrt{h}-\frac{M}{2}h^{(\alpha-1)/2}+o(h^{(\alpha-1)/2}) \\
\mbox{ and } \phi_+(h)&=&\sqrt{h}+\frac{m}{2}h^{(\alpha-1)/2}+o(h^{(\alpha-1)/2})\end{eqnarray*} when $h\to 0$ for some constants $m>0$ and $M>0$ which are respectively a (specific) lower and upper bound for $\varepsilon$.
\end{lemma}
\begin{proof}
By assumption $\lvert\varepsilon\rvert\leq M$, hence we have: 
\begin{align*}
\{x\in B(0,r) & \mid\lVert x-x^{\star}\rVert_{\mathbf{H}}^{2}+M\lVert x-x^{\star}\rVert_{\mathbf{H}}^{\alpha}\leq h\}\subset S_{h}
\end{align*}
Let $g\colon u\mapsto u^{2}+Mu^{\alpha}$. This is a continuous,
strictly increasing function on $[0,+\infty)$. By a classical consequence
of the intermediate value theorem, this implies that $g$ admits a
continuous, strictly increasing inverse function. Note that $g(0)=0$
hence $g^{-1}(0)=0$. Thus we can write $\{u\geq 0|u^{2}+Mu^{\alpha}\le h\}=[0,g^{-1}(h)]$.
We now denote $g^{-1}$ by $\phi_{-}$. As $\phi_{-}$ is non-decreasing,
we get
\begin{align*}
\{x\in B(0,r) & \mid\lVert x-x^{\star}\rVert_{\mathbf{H}}^{2}+M\lVert x-x^{\star}\rVert_{\mathbf{H}}^{\alpha}\leq h\}=A_{h}\cap B(0,r)
\end{align*}
Now observe that for $h$ sufficiently small
\[
\{x\in B(0,r)\mid\lVert x-x^{\star}\rVert_{\mathbf{H}}^{2}+M\lVert x-x^{\star}\rVert_{\mathbf{H}}^{\alpha}\leq h\}=A_{h}.
\]
Indeed, if $x\in A_{h}$, we have by the triangle inequality and~\eqref{eq:sym-mat-ineq}
\begin{align*}
\lVert x\rVert & \le\lVert x^{\star}\rVert+\lVert x-x^{\star}\rVert\\
 & \le\lVert x^{\star}\rVert+e_{1}(\mathbf{H})^{-1/2}\lVert x-x^{\star}\rVert_{\mathbf{H}}\\
 & \le\lVert x^{\star}\rVert+e_{1}(\mathbf{H})^{-1/2}\phi_{-}(h)
\end{align*}
Recall that by assumption $\lVert x^{\star}\rVert<r$ and let $\delta=r-\lVert x^{\star}\rVert$.
As $\phi_{-}(h)\rightarrow_{h\rightarrow0}0$, for $h$ sufficiently
small, we have $e_{1}(\mathbf{H})^{-1/2}\phi_{-}(h)\le\delta$ hence
$\lVert x\rVert\le r$ for $h$ sufficiently small, which gives the inclusion $A_h \subset S_h$.\\
For the asymptotics of $\phi_{-}$, as we have by definition $\phi_{-}(h)^{2}(1+M\phi_{-}(h)^{\alpha-2})=h$,
and as $\phi_{-}(h)\rightarrow_{h\rightarrow0}0$ we deduce that $\phi_{-}(h)\sim_{0}\sqrt{h}$.
Let us define $u(h)=\phi_{-}(h)-\sqrt{h}$. We have $u(h)\in o(\sqrt{h})$.
We then compute: 
\begin{align*}
(\sqrt{h}+u(h))^{2}+M(\sqrt{h}+u(h))^{\alpha}=h
\end{align*}
This gives
\begin{align*}
u(h)(u(h)+2\sqrt{h}) & =-Mh^{\alpha/2}(1+\frac{u(h)}{\sqrt{h}})^{\alpha}\\
u(h)(\frac{u(h)}{2\sqrt{h}}+1) & =-\frac{M}{2}h^{(\alpha-1)/2}(1+\frac{u(h)}{\sqrt{h}})^{\alpha}
\end{align*}
As $u(h)\in o(\sqrt{h})$ for $h\rightarrow0$, we obtain
\[
u(h)\sim-\frac{M}{2}h^{(\alpha-1)/2}.
\]
which concludes for $\phi_{-}$.

On the other side, we recall that $f(x)>0$ for all $x\neq x^{\star}$
as $x^{\star}$ is the unique minimum of $f$ on $B(0,r)$. Write
\[
0<\lVert x-x^{\star}\rVert_{\mathbf{H}}^{2}(1+\lVert x-x^{\star}\rVert_{\mathbf{H}}^{\alpha-2}\varepsilon(x-x^{\star})).
\]
Now observe that, as $\lVert x^{\star}\rVert<r$, we have for $x\in B(0,r)$,
by the triangle inequality, $\lVert x-x^{\star}\rVert<2r$. Hence,
by the classical inequality for the $\mathbf{H}$-norm~\eqref{eq:sym-mat-ineq}, we get
\begin{align*}
\varepsilon(x-x^{\star}) & >-\frac{1}{\lVert x-x^{\star}\rVert_{\mathbf{H}}^{\alpha-2}}\geq-\left(\sqrt{e_{d}(\mathbf{H})}2r\right)^{-(\alpha-2)}=:-m
\end{align*}
So we have: 
\begin{align*}
S_{h}\subset\{x\in B(0,r) & \mid\lVert x-x^{\star}\rVert_{\mathbf{H}}^{2}-m\lVert x-x^{\star}\rVert_{\mathbf{H}}^{\alpha}\leq h\}
\end{align*}
The function $g\colon u\mapsto u^{2}-mu^{\alpha}$ is differentiable.
A study of the derivative shows that $g$ is continuous, strictly
increasing on $[0,r_{0}]$ and continuous, strictly decreasing on
$[r_{0},+\infty[$ where $r_{0}=(\frac{2}{\alpha m})^{1/(\alpha-2)}$.
Hence $g_{|[0,r_{0}]}$ admits a continuous strictly increasing inverse
$\phi_{+}$ and $g_{|[r_{0},+\infty[}$ a continuous strictly decreasing
inverse $\Tilde{\phi}$. We thus write 
\[
\{u\ge0|u^{2}-mu^{\alpha}\le h\}=[0,\phi_{+}(h)]\cup[\Tilde{\phi}(h),+\infty).
\]
Hence 
\begin{align*}
 \{x\in B(0,r)\mid&\lVert x-x^{\star}\rVert_{\mathbf{H}}^{2}-m\lVert x-x^{\star}\rVert_{\mathbf{H}}^{\alpha}\leq h\}\\
&=\big(B_{h}\cap B(0,r)\big)\cup\big(B(0,r)\cap V_{h}\big)
\end{align*}

with $V_{h}=\{x\in\mathbb{R}^{d}|\ \lVert x-x^{\star}\rVert_{\mathbf{H}}>\tilde{\phi}(h)\}$.
We now show that for $h$ sufficiently small
\[
\{x\in B(0,r)\mid\lVert x-x^{\star}\rVert_{\mathbf{H}}^{2}-m\lVert x-x^{\star}\rVert_{\mathbf{H}}^{\alpha}\leq h\}=B_{h}.
\]
Indeed, note first that if $x\in B(0,r)$, we obtain by~\eqref{eq:sym-mat-ineq}
\[
\lVert x-x^{\star}\rVert_{\mathbf{H}}^{2}\le e_{d}(\mathbf{H})\lVert x-x^{\star}\rVert^{2}<4e_{d}(\mathbf{H})r^{2}.
\]
where we have used that, as $\lVert x\rVert<r$, the triangle inequality
gives $\lVert x-x^{\star}\rVert<2r$. Hence $B(0,r)\subset\{x\in\mathbb{R}^{d}|\ \lVert x-x^{\star}\rVert_{\mathbf{H}}^{2}<4e_{d}(\mathbf{H})r^{2}\}$.
We now show that $B(0,r)\subset\{x\in\mathbb{R}^{d}|\ \lVert x-x^{\star}\rVert_{\mathbf{H}}\le\Tilde{\phi}(h)\}$.
Indeed, at $h=0$, $0=\phi_{+}(0)<\Tilde{\phi}(0)$ are by definition,
the two roots of 
\[
u^{2}-mu^{\alpha}=0.
\]
Hence $\Tilde{\phi}(0)=\sqrt{e_{d}(\mathbf{H})2r}$. By continuity
of $\Tilde{\phi}(h)$ at $h=0$, we obtain that $B(0,r)\subset\{x\in\mathbb{R}^{d}|\ \lVert x-x^{\star}\rVert_{\mathbf{H}}\le\Tilde{\phi}(h)\}$
for $h$ sufficiently small. As $\phi_{+}(h)\le\Tilde{\phi}(h)$,
we thus obtain that, for $h$ sufficiently small, $V_{h}\cap B(0,r)=\emptyset$.
Next, the same line of reasoning as the one for $\phi_{-}$, using
that $\phi_{+}(h)\rightarrow_{h\rightarrow0}0$ and $\lVert x^{\star}\rVert<r$,
shows that $B_{h}\cap B(0,r)=B_{h}$ for $h$ sufficiently small.
\\
Hence, for $h$ small enough we have
\[
\{x\in B(0,r)\mid\lVert x-x^{\star}\rVert_{\mathbf{H}}^{2}-m\lVert x-x^{\star}\rVert_{\mathbf{H}}^{\alpha}\leq h\}=B_{h}.
\]
This gives $S_h \subset B_h$.\\
Finally, similarly to $\phi_{-}$, we can show that $\phi_{+}(h)=\sqrt{h}+\frac{m}{2}h^{(\alpha-1)/2}+o(h^{(\alpha-1)/2})$,
which concludes the proof of this lemma. 
\end{proof}

\section{Bounds for random search}
\label{sec:randomsearch}
In this section we provide upper bounds and lower bounds for the random search algorithm for functions satisfying Assumption~\ref{ass:principal}. These bounds will also be useful for analyzing the convergence of the $\mu$-best approach.
\subsection{Upper bound}
First, we prove an upper bound for functions satisfying Assumption~\ref{ass:principal}.

\begin{lemma}[Upper bound for random search algorithm]\label{lem:upper1best}
Let $f$ be a function satisfying Assumption~\ref{ass:principal}. There exists a constant $C_0>0$ and an integer $\lambda_0\in \mathbb{N}$ such that for all integers $ \lambda\geq \lambda_0$:
$$\mathbb{E}_{X_1,\cdots,X_\lambda\sim U(B(0,r))}\left[ f\left(X_{(1)}\right)\right] \leq C_0 \lambda^{-\frac2d}\quad. $$
\end{lemma}
\begin{proof}
Let us first recall the following classical property about the expectation of a positive valued random variable:    
\begin{align*}
\mathbb{E}_{X_1,\dots,X_\lambda\sim U(B(0,r))}&\left[ f\left(X_{(1)}\right)\right]= \int_0^\infty \mathbb{P}\left[ f\left(X_{(1)}\right)\geq t\right] dt
\end{align*}
By independence of the samples we have:  
\begin{align*}
\int_0^\infty \mathbb{P}\left[ f\left(X_{(1)}\right)\geq t\right] dt= \int_0^\infty \mathbb{P}_{X\sim U(B(0,r))}\left[ f\left(X\right)\geq t\right]^\lambda dt
\end{align*}
Then thanks to Lemma~\ref{lem:sandwich}:
\begin{align*}
     \int_0^\infty \mathbb{P}_{X\sim U(B(0,r))}&\left[ f\left(X\right)\geq t\right]^\lambda dt\\
 &\leq \int_0^\infty \mathbb{P}_{X\sim U(B(0,r))}\left[ L\lVert X-x^\star\rVert^2 \geq t\right]^\lambda dt\\
  &= \int_0^{L\left(r+\lVert x^\star\rVert\right)^2} \mathbb{P}\left[ \lVert X-x^\star\rVert \geq \sqrt{\frac{t}{L}}\right]^\lambda dt
\end{align*}
where the second equality follows because $\lVert X-x^{\star} \rVert \leq r$ almost surely.
Then, by definition of the uniform law as well as the non-increasing character of $t\mapsto \mathbb{P}_{X\sim U(B(0,r))}\left[ \lVert X-x^\star\rVert \geq \sqrt{\frac{t}{L}}\right]$, we obtain
\begin{align*}
  &\int_0^{L\left(r+\lVert x^\star\rVert\right)^2} \mathbb{P}_{X\sim U(B(0,r))}\left[ \lVert X-x^\star\rVert \geq \sqrt{\frac{t}{L}}\right]^\lambda dt \\
&=\int_0^{L\left(r-\lVert x^\star\rVert\right)^2} \mathbb{P}_{X\sim U(B(0,r))}\left[ \lVert X-x^\star\rVert \geq \sqrt{\frac{t}{L}}\right]^\lambda dt \\
&+ \int_{L\left(r-\lVert x^\star\rVert\right)^2}^{L\left(r+\lVert x^\star\rVert\right)^2} \mathbb{P}_{X\sim U(B(0,r))}\left[ \lVert X-x^\star\rVert \geq \sqrt{\frac{t}{L}}\right]^\lambda dt\\
&\leq \int_0^{L\left(r-\lVert x^\star\rVert\right)^2} \left[ 1-\left(\sqrt{\frac{t}{Lr^2}}\right)^d\right]^\lambda dt\\
 &+L\left(\left(r+\lVert x^\star\rVert\right)^2-\left(r-\lVert x^\star\rVert\right)^2\right) \mathbb{P}\left[ \lVert X-x^\star\rVert \geq r-\lVert x^\star\rVert\right]^\lambda \\
&\leq \int_0^{Lr^2} \left[ 1-\left(\frac{t}{Lr^2}\right)^{\frac{d}{2}}\right]^\lambda dt+4Lr\lVert x^\star\rVert\mathbb{P}\left[ \lVert X-x^\star\rVert \geq r-\lVert x^\star\rVert\right]^\lambda\\
&= Lr^2\int_0^{1} \left[ 1-u^{\frac{d}{2}}\right]^\lambda du+4Lr\lVert x^\star\rVert\mathbb{P}\left[ \lVert X-x^\star\rVert \geq r-\lVert x^\star\rVert\right]^\lambda
\end{align*}
Note that  $\mathbb{P}\left[ \lVert X-x^\star\rVert < r-\lVert x^\star\rVert\right] < 1$. Thus the second term in the last equality satisfies $\mathbb{P}\left[ \lVert X-x^\star\rVert < r-\lVert x^\star\rVert\right]^\lambda
\in o(\lambda^{-2/d})$.
The first term has a closed form given in~\cite{ppsnkbest}:
\begin{align*}
    \int_0^{1} \left[ 1-u^{\frac{d}{2}}\right]^\lambda du=\frac{\Gamma(\frac{d+2}{d})\Gamma(\lambda+1)}{\Gamma(\lambda+1+2/d)}
\end{align*}
Finally thanks to the Stirling approximation, we conclude:
\begin{align*}
\mathbb{E}_{X_1,\dots,X_\lambda\sim U(B(0,r))}&\left[ f\left(X_{(1)}\right)\right]\leq C_1 \lambda^{-2/d}+o(\lambda^{-2/d})
\end{align*}
where $C_1>0$ is a constant independent from $\lambda$.
\end{proof}
This lemma proves that the strategy consisting in returning the best sample (i.e. random search) has an upper rate of convergence of order $\lambda^{-2/d}$, which depends on dimension of the space. It also worth noting this result is common in the literature~\cite{bach,bergstra2012random}

\subsection{Lower bound}
We now give a lower bound for the convergence of the random search algorithm. We also prove a conditional expectation bound that will be useful for the analysis of the $\mu$-best averaging approach.

\begin{lemma}[Lower bound for random search algorithm]\label{lem:lower1best}
Let $f$ be a function satisfying Assumption~\ref{ass:principal}. There exist a constant $C_1>0$ and $\lambda_1\in \mathbb{N}$ such that for all integers $\lambda\geq \lambda_1$, we have the following lower bound for random search:
\begin{align*}
\mathbb{E}_{X_1,\dots,X_\lambda\sim U(B(0,r))}&\left[ f\left(X_{(1)}\right)\right]\geq C_1 \lambda^{-2/d}\quad.
\end{align*}
Moreover, let $(\mu_\lambda)_{\lambda\in\mathbb{N}}$ be a sequence of integers such that $\forall\lambda\geq 2$, $1\leq \mu_\lambda \leq \lambda -1$ and $\mu_\lambda\to\infty$. Then, there exist a constant $C_2>0$ and $\lambda_2\in \mathbb{N}$ such that for all $h\in [0,\max f]$ and $\lambda\geq\lambda_2$, we have the following lower bound when the sampling is conditioned: 
\begin{align*}
   \mathbb{E}_{X_1,\dots,X_\lambda\sim U(B(0,r))}&\left[ f\left(X_{(1)}\right)\mid f(X_{(\mu_{\lambda}+1)})= h\right] \geq C_2 h\mu_\lambda^{-2/d}\quad.
\end{align*}
\end{lemma}
\begin{proof}
The proof is very similar to the previous one. Let us first show the unconditional inequality. We use the identity for the expectation of a positive random variable
\begin{align*}
\mathbb{E}&_{X_1,\dots,X_\lambda\sim  U(B(0,r))}\left[ f\left(X_{(1)}\right)\right] \\
&= \int_0^\infty \mathbb{P}_{X_1,\dots,X_\lambda\sim  U(B(0,r))}\left[ f\left(X_{(1)}\right)\geq t\right] dt
\end{align*}
Since the samples are independent, we have
\begin{align*}
 \int_0^\infty &\mathbb{P}_{X_1,\dots,X_\lambda\sim U(B(0,r))}\left[ f\left(X_{(1)}\right)\geq t \right] dt\\ 
 & = \int_0^\infty \mathbb{P}_{X\sim  U(B(0,r))}\left[ f\left(X\right)\geq t\right]^\lambda dt
 \end{align*}
 Using Lemma~\ref{lem:sandwich}, we get:
 \begin{align*}
 \int_0^\infty& \mathbb{P}_{X\sim U(B(0,r))}\left[ f\left(X\right)\geq t\right]^\lambda dt\\
& \geq \int_0^\infty \mathbb{P}_{X\sim  U(B(0,r))}\left[ l\lVert X-x^\star\rVert^2 \geq t\right]^\lambda dt\\
& \geq \int_0^{l(r-\lVert x^\star\rVert)^2} \mathbb{P}_{X\sim U(B(0,r))}\left[ l\lVert X-x^\star\rVert^2 \geq t\right]^\lambda dt\\
&=\int_0^{l(r-\lVert x^\star\rVert)^2} \left[ 1-\left(\sqrt{\frac{t}{lr^2}}\right)^d\right]^\lambda dt
\end{align*}
We can decompose the integral to obtain:
\begin{align*}
    &\int_0^{l(r-\lVert x^\star\rVert)^2} \left[ 1-\left(\sqrt{\frac{t}{lr^2}}\right)^d\right]^\lambda dt\\
    &=\int_0^{lr^2} \left[ 1-\left(\sqrt{\frac{t}{lr^2}}\right)^d\right]^\lambda - \int_{l(r-\lVert x^\star\rVert)^2}^{lr^2} \left[ 1-\left(\sqrt{\frac{t}{lr^2}}\right)^d\right]^\lambda dt\\
    &\geq lr^2\frac{\Gamma(\frac{d+2}{d})\Gamma(\lambda+1)}{\Gamma(\lambda+1+\frac2d)}- l(r^2-(r-\lVert x^\star\rVert)^2)\left[ 1-\left(\frac{r-\lVert x^\star\rVert}{r}\right)^{d}\right]^\lambda\\
    &\geq \frac12 lr^2 \Gamma(\frac{d+2}{d}) \lambda^{-2/d}\text{ for $\lambda$ sufficiently large.}
\end{align*}
where the last inequality follows by Stirling's approximation applied to the first term and because the second term is $o(\lambda^{-2/d})$ as in previous proof.\\
This concludes the proof of the first part of the lemma. Let us now treat the case of the conditional inequality. Using the same first identity as above we have
\begin{align*}
\mathbb{E}&_{X_1,\dots,X_\lambda\sim U(B(0,r))}\left[ f\left(X_{(1)}\right)\mid f(X_{(\mu_{\lambda}+1)}) = h\right] \\
&= \int_0^\infty \mathbb{P}_{X_1,\dots,X_\lambda\sim U(B(0,r))}\left[ f\left(X_{(1)}\right)\geq t \mid f(X_{({\mu_{\lambda}}+1)}) = h\right] dt
\end{align*}
%
\begin{rmq}
\label{rk:samples}
Note that if we sample $\lambda$ independent variables $X_1 \ldots X_\lambda$ while conditioning on $f(X_{(\mu+1)})=h$ 
and keep only the $\mu$-best variables $X_i$ such that $f(X_i)\le h$, this is exactly equivalent to sampling directly $X_1 \ldots X_\mu$ from the $h$-level set. This result was justified and used in~\cite{ppsnkbest} in their proofs.
\end{rmq}
Hence we obtain
\begin{align*}
 \int_0^\infty &\mathbb{P}_{X_1,\dots,X_\lambda\sim U(B(0,r))}\left[ f\left(X_{(1)}\right)\geq t \mid f(X_{(\mu_{\lambda}+1)})= h\right] dt\\ 
 & = \int_0^\infty \mathbb{P}_{X\sim U(S_h)}\left[ f\left(X\right)\geq t\right]^{\mu_{\lambda}} dt
 \end{align*}
 Using Lemma~\ref{lem:sandwich}, we get:
 \begin{align*}
 \int_0^\infty& \mathbb{P}_{X\sim U(S_h)}\left[ f\left(X\right)\geq t\right]^{\mu_{\lambda}} dt\\
& \geq \int_0^\infty \mathbb{P}_{X\sim U(S_h)}\left[ l\lVert X-x^\star\rVert^2 \geq t\right]^{\mu_{\lambda}} dt\\
& \geq \int_0^\infty \mathbb{P}_{X\sim U(B(x^\star,\sqrt{\frac{h}{l}}))}\left[ l\lVert X-x^\star\rVert^2 \geq t\right]^{\mu_{\lambda}} dt
\end{align*}
where the last inequality follows from the inclusion $S_h\subset B(x^\star,\sqrt{\frac{h}{l}})$, which is also a consequence of Lemma~\ref{lem:sandwich}. We then get
\begin{align*}
\int_0^\infty& \mathbb{P}_{X\sim U(B(x^\star,\sqrt{\frac{h}{l}}))}\left[ l\lVert X-x^\star\rVert^2 \geq t\right]^{\mu_{\lambda}} dt\\
& = \int_0^h \mathbb{P}_{X\sim U(B(x^\star,\sqrt{\frac{h}{l}}))}\left[ l\lVert X-x^\star\rVert^2 \geq t\right]^{\mu_{\lambda}} dt\\
& = \int_0^{h} \left[ 1-\left(\sqrt{\frac{t}{h}}\right)^d\right]^{\mu_{\lambda}} dt\\
& = h\frac{\Gamma(\frac{d+2}{d})\Gamma(\mu_\lambda+1)}{\Gamma(\mu_\lambda+1+2/d)}\\
& \geq \frac12 h \Gamma(\frac{d+2}{d})\mu_\lambda^{-2/d}\text{ for $\lambda$ sufficiently large.}
\end{align*}
\end{proof}
This lemma, along with Lemma~\ref{lem:upper1best}, proves that for any function satisfying Assumption~\ref{ass:principal}, its rate of convergence is exponentially dependent on the dimension and of order $\lambda^{-2/d}$ where $\lambda$ is the number of points sampled to estimate the optimum. 
\begin{rmq}[Convergence of the distance to the optimum]
It is worth noting that, thanks to Lemma~\ref{lem:sandwich}, the convergence rates are also valid for the square distance to the optimum $x^\star$. 
\end{rmq}

\section{Convergence rates for the $\mu$-best averaging approach}
\label{sec:mubestrate}
In the next section we focus on the case where we average the $\mu$ best samples among the $\lambda$ samples.  We first prove a lemma when the sampling is conditional on the $(\mu+1)$-th value.
\begin{lemma} 
\label{lem:condit-upper} 
Let $f$ be a function satisfying Assumption~\ref{ass:principal}. There exists a constant $C_3>0$ such that for all $h\in[0,\max f]$ and $\lambda$ and $\mu$ two integers such that $1\leq \mu \leq \lambda-1$, we have the following conditional upper bound:
\[
\mathbb{E}_{X_{1},...X_{\lambda}\sim U(B(0,r))}\left[f(\bar{X}_{(\mu)})|f(X_{(\mu+1)})=h\right]\le C_3\left(\frac{h}{\mu}+h^{\alpha-1}\right).
\]

\end{lemma} 
\begin{proof}
We first decompose the expectation as follows.
\begin{align}
\mathbb{E}&_{X_{1},...X_{\lambda}\sim U(B(0,r))}  \left[f(\bar{X}_{(\mu)})|f(X_{(\mu+1)}))=h\right]\nonumber\\
 & =\mathbb{E}_{X_{1},...X_{\mu}\sim U(S_{h})}\left[f(\bar{X}_{\mu})\right]\nonumber\\
 & =\mathbb{E}_{X_{1},\cdots,X_{\mu}\sim U(S_{h})}\left[\lVert \bar{X}_{\mu}-x^\star\rVert_{\mathbf{H}}^2\right]\label{eq:first-term}\\
 &+\mathbb{E}_{X_{1},\cdots,X_{\mu}\sim U(S_{h})}\left[\lVert\bar{X}_{\mu}-x^\star\rVert^{\alpha}_{\mathbf{H}}\varepsilon(\bar{X}_{\mu}-x^\star)\right]\label{eq:snd-term}
\end{align}
where we have use the same argument as in Remark~\ref{rk:samples} in the first equality.
We will treat the terms~\eqref{eq:first-term} and~\eqref{eq:snd-term} independently. We first look at~\eqref{eq:first-term}. We have the following ``bias-variance'' decomposition.
\begin{align*}
\mathbb{E}_{X_{1},\cdots,X_{\mu}\sim U(S_{h})}\lVert\bar{X}_{\mu}-x^{\star}\rVert^{2}_{\mathbf{H}} 
=&(1-\frac{1}{\mu})\lVert\mathbb{E}_{X\sim U(S_{h})}X-x^{\star}\rVert^{2}_{\mathbf{H}}\\
&+\frac{1}{\mu}\mathbb{E}_{X\sim U(S_{h})}\lVert X-x^{\star}\rVert^{2}_{\mathbf{H}}
\end{align*}
We will use Lemma~\ref{lemma:sandwich-set}. We have $A_h\subset S_h\subset B_h$. Hence for the variance term
\[
\frac{1}{\mu}\mathbb{E}_{X\sim U(S_{h})}\lVert X-x^{\star}\rVert^{2}_{\mathbf{H}}\leq\frac{1}{\mu}\mathbb{E}_{X\sim U(S_{h})} \phi_+(h)^{2} \leq \frac{\phi_+(h)^{2}}{\mu}\sim_{0} \frac{h}{\mu}.
\]
where $\sim_0$ means ''is equivalent to $\dots$ when $h\to0$, in other words, $u(h)\sim_0 v(h)$ iff $\frac{u(h)}{v(h)}\to 0$ as $h\to0$. For the bias term, recall that \[\mathbb{E}_{X\sim U(S_{h})}\left[X-x^\star\right] = \frac{1}{\mathrm{vol}(S_h)}\int_{S_h} (x-x^\star) dx.\] We then have by inclusion of sets
\begin{align*}
    \mathrm{vol}(A_h)\leq \mathrm{vol}(S_h)\leq \mathrm{vol}(B_h)
\end{align*}
Note that the volume of the $d$-dimensional ellipsoid $B_h$ satisfies $\mathrm{vol}(B_h)=\phi_+(h)^{d}\frac{\omega_{d}}{\mathrm{det}(\mathbf{H})}$ with $\omega_{d}=\mathrm{vol}(B(0,1))$ and similarly for $A_h$.
From this we deduce by the squeeze theorem that 
\[\mathrm{vol}(S_h)\sim \frac{\omega_dh^{d/2}}{\mathrm{det}(\mathbf{H})}.\]We now decompose the integral 
\begin{align*}
\int_{S_{h}}(x-x^\star)dx & =\int_{A_{h}}(x-x^\star)dx+\int_{S_{h}\setminus A_{h}}(x-x^\star)dx\\
 & =\int_{S_{h}\setminus A_{h}}(x-x^\star)dx
\end{align*}
(because $A_{h}$ is an ellipsoid centered at $x^\star$ hence the integral of $x-x^\star$
over it is $0$). We then upper-bound using the triangle inequality for the $\mathbf{H}-$norm:
\begin{align*}
\lVert\int_{{S}_{h}\setminus A_{h}}(x-x^\star)dx\rVert_{\mathbf{H}} &\leq \int_{{S}_{h}\setminus A_{h}}\lVert x-x^\star\rVert_{\mathbf{H}} dx\\
& \le \phi_+(h)\mathrm{vol}({S}_{h}\setminus A_{h})\\
 & = \phi_+(h)(\mathrm{vol}({S}_{h})-\mathrm{vol}(A_{h}))\\
 &\leq \phi_+(h)(\mathrm{vol}(B_h)-\mathrm{vol}(A_h))\\
 & \sim d\frac{\omega_d}{\mathrm{det}(\mathbf{H})}\frac{m+M}{2}h^{d/2}h^{(\alpha-1)/2}
\end{align*}
For the last equivalent, we used a Taylor expansion for the volume of $A_h$ and $B_h$. We conclude that there exist $h_1>0$ and a constant $C>0$ not depending on $\lambda$ and $\mu$ such that for $h\leq h_1$,
\begin{align*}
\lVert\mathbb{E}_{X\sim U(S_{h})}\left[X\right]-x^\star\rVert^{2}_{\mathbf{H}} \leq C h^{\alpha-1}
\end{align*}
Since $h$ is upper bounded by $\max f$, the previous inequality can be extended to $h\in[0,\max f]$, with a possibly larger constant still not depending on $\lambda$ and $\mu$. Let us now upper bound the remainder term~\eqref{eq:snd-term}.
As $\varepsilon \leq M$ by assumption, we can write
\begin{align*}
\mathbb{E}_{X_{1},\cdots,X_{\mu}\sim U(S_{h})}&\left[\lVert\bar{X}_{\mu}-x^\star\rVert^{\alpha}_{\mathbf{H}}\varepsilon(\bar{X}_{\mu}-x^\star)\right]\\
&\le M\mathbb{E}_{X_{1},\cdots,X_{\mu}\sim U(S_{h})}\left[\lVert\bar{X}_{\mu}-x^\star\rVert^{\alpha}_{\mathbf{H}}\right]
\end{align*}

We have $X_{1},\cdots,X_{\mu}\in S_{h}\subset B_{h}$
hence by the convexity of $B_h$ (which is a ball for the $\mathbf{H}$-norm) we also have $\bar{X}_{\mu}\in B_{h}$ and
thus, for $h$ sufficiently small, we have:
\[
\lVert\bar{X}_{\mu}-x^\star\rVert_{\mathbf{H}}\le \phi_+(h).
\]
Note that $\phi_+(h)\sim_{0} \sqrt{h} $ thus, for $h$ sufficiently small, $ \lVert\bar{X}_{\mu}-x^\star\rVert_{\mathbf{H}}\le 1$
almost surely, hence, as $\alpha > 2$
\begin{align*}
\lVert\bar{X}_{\mu}-x^\star\rVert^{\alpha}_{\mathbf{H}}\leq \lVert\bar{X}_{\mu}-x^\star\rVert^{2}_{\mathbf{H}}
\end{align*}
almost surely. Since $h$ is upper bounded, we have the existence of a constant $C'>0$ not depending on $\lambda$ and $\mu$, such that for all $h\in [0,\max f]$,
\begin{align*}
\lVert\bar{X}_{\mu}-x^\star\rVert^{\alpha}_{\mathbf{H}}\leq C'\lVert\bar{X}_{\mu}-x^\star\rVert^{2}_{\mathbf{H}}
\end{align*}
Thus we can upper bound the remainder with the same bounds as the one for the main term (up to constants), for any $h\in[0,\max f]$.
We now group the ``main'' term and remainder term to get the existence of a constant $C_3>0$ not depending on $\lambda$ and $\mu$ such that for all $h\in [0,\max f]$,
\[
\mathbb{E}_{X_{1},...X_{\lambda}\sim U(B(0,r))}\left[f(\bar{X}_{(\mu)})|f(X_{(\mu+1)})=h\right]\le C_3\left(\frac{h}{\mu}+h^{\alpha-1}\right)\quad.
\]

\end{proof}


We are now set to prove our main result, which is an upper convergence rate for the $\mu$-best approach. This is the main result of the paper.
\begin{thm}\label{thm:principal}
Let $f$ be a function satisfying Assumption~\ref{ass:principal}.  Let $(\mu_\lambda)_{\lambda\in\mathbb{N}}$ be a sequence of integers such that $\forall\lambda\geq 2$, $1\leq \mu_\lambda \leq \lambda -1$ and $\mu_\lambda\to\infty$. Then, there exist two constants $C,C'>0$ and $\Tilde{\lambda}\in \mathbb{N}$ such that  for $\lambda\geq\Tilde{\lambda}$, we have the upper bound: 
\begin{align*}
 \mathbb{E}_{X_{1},\dots,X_{\lambda}\sim U(B(0,r))}\left[f(\bar{X}_{(\mu_\lambda)})\right] \leq C\frac{\mu_\lambda^{\frac{2(\alpha-1)}{d}}}{\lambda^{\frac{2(\alpha-1)}{d}}}+C'\frac{\mu_\lambda^{\frac{2}{d}-1}}{\lambda^{\frac{2}{d}}}\quad.
\end{align*}
In particular if $\mu_{\lambda}\sim C^{''}\lambda^{\frac{2(\alpha-2)}{d+2(\alpha-2)}}$
for some $C^{''}>0$, we obtain:
\begin{align*}
\mathbb{E}_{X_{1},\dots,X_{\lambda}\sim U(B(0,r))}\left[f(\bar{X}_{(\mu_{\lambda})})\right] & \le C^{'''}\lambda^{-\frac{2(\alpha-1)}{d+2(\alpha-2)}}
\end{align*}
for some $C'''>0$ independent of $\lambda$.\end{thm}

\newotc{We note that $\frac{\mu}{\lambda}\to 0$ as $\lambda\to 0$. This makes sense intuitively: we average points in a sublevel set, which makes sense only if, asymptotically in $\lambda$, this sublevel set shrinks to a neighborhood of the optimum.}

\begin{proof}
The random variable $f(X_{(\mu_\lambda+1)})$ takes its values in $[0,\max f]$ almost surely. As such, thanks to Lemma~\ref{lem:condit-upper}, there exists a constant $C_3>0$ such that for all $\lambda\geq 1$:
 \begin{align*}
      \mathbb{E}\left[f(\bar{X}_{(\mu_\lambda)})\right]&=\mathbb{E}\left[\mathbb{E}\left[f(\bar{X}_{(\mu_\lambda)})\mid f(X_{(\mu_\lambda+1)}) \right]\right]\\
        &\leq \mathbb{E}\left[C_3\left(\frac{1}{\mu_\lambda}f(X_{(\mu_\lambda+1)})+f(X_{(\mu_\lambda+1)})^{\alpha-1}\right)\right]\\
      &= C_3\left(\frac{1}{\mu_\lambda}\mathbb{E}\left[f(X_{(\mu_\lambda+1)})\right]+\mathbb{E}\left[f(X_{(\mu_\lambda+1)})^{\alpha-1}\right]\right)
 \end{align*}
 
Let us first bound $\mathbb{E}\left[f(X_{(\mu_\lambda+1)})\right]$. Thanks to Lemma~\ref{lem:lower1best}, there exist a constant $C_2>0$ and $\lambda_2\in \mathbb{N}$ such that:
\begin{align*}
  \mathbb{E} \left[f(X_{(\mu_\lambda+1)})\right]&\leq \frac{\mu_\lambda^{2/d}}{C_2}\mathbb{E}\left[\mathbb{E} \left[f(X_{(1)})\mid f(X_{(\mu_\lambda+1)})\right]\right]\\
  &=\frac{\mu_\lambda^{2/d}}{C_2}\mathbb{E} \left[f(X_{(1)})\right]
\end{align*}

Thanks to Lemma~\ref{lem:upper1best}, there exists a constant $C_0>0$ and an integer $\lambda_0\in \mathbb{N}$ such that for all integers $ \lambda\geq \lambda_0$:
$$\mathbb{E}_{X_1,\cdots,X_\lambda\sim U(B(0,r))}\left[ f\left(X_{(1)}\right)\right] \leq C_0 \lambda^{-\frac2d}\quad. $$

Then finally for $\lambda\geq\max(\lambda_0,\lambda_2)$
\begin{align*}
    \mathbb{E} \left[f(X_{(\mu_\lambda+1)})\right]\leq\frac{C_0}{C_2} \frac{\mu_\lambda^{2/d}}{\lambda^{2/d}}\quad.
\end{align*}

For the term $\mathbb{E}\left[f(X_{(\mu_{\lambda}+1)})^{\alpha-1}\right]$,
we write thanks to Lemma~\ref{lem:lower1best}
\[
\mathbb{E}\left[f(X_{(\mu_{\lambda}+1)})^{\alpha-1}\right]\leq\frac{\mu_{\lambda}^{2(\alpha-1)/d}}{C_{2}^{\alpha-1}}\mathbb{E}\left[\mathbb{E}\left[f(X_{(1)})\mid f(X_{(\mu_{\lambda}+1)})\right]^{\alpha-1}\right].
\]
Then, by Jensen's inequality for the conditional expectation, we get
\[
\mathbb{E}\left[f(X_{(\mu_{\lambda}+1)})^{\alpha-1}\right]\le\frac{\mu_{\lambda}^{2(\alpha-1)/d}}{C_{2}^{\alpha-1}}\mathbb{E}\left[f(X_{(1)})^{\alpha-1}\right].
\]
Similarly to Lemma~\ref{lem:upper1best}, by replacing $\lVert X-x^\star\rVert^2$ by $\lVert X-x^\star\rVert^{2(\alpha-1)}$, one can show $\mathbb{E}\left[f(X_{(1)})^{\alpha-1}\right]\le C'_{3}\lambda^{-2(\alpha-1)/d}$
for some $C'_{3}>0$ independent of $\lambda$. We thus get $\mathbb{E}\left[f(X_{(\mu_{\lambda}+1)})^{\alpha-1}\right]\le C\frac{\mu_{\lambda}^{2(\alpha-1)/d}}{\lambda^{2(\alpha-1)/d}}$
for some $C>0$ independent of $\lambda$, which, combined with the
above bound on $\mathbb{E}\left[f(X_{(\mu_{\lambda}+1)})\right]$,
concludes the proof of the main bound.\\ To conclude for the final bound, it suffices to notice that this choice of $\mu_{\lambda}$ ensures that the two terms in the upper bound are of the same order.\end{proof}

This theorem gives an asymptotic upper rate of convergence for the algorithm that consists in averaging the best samples to optimize a function with parallel evaluations. The proof of the optimality of the rate is left as further work. We also remark that the selection ratio depends on the dimension and goes to $0$ as $\lambda\to\infty$. It sounds natural since the level sets might be assymetric and then keeping a constant selection rate would give a biased estimate of the optimum (see Figure~\ref{fig:levelset}). However, the choice proposed for $\mu$ is the best one can make with regards to the upper bound we obtained. We make two important remarks about the theorem.
\begin{figure}
	\includegraphics[width=.4\textwidth]{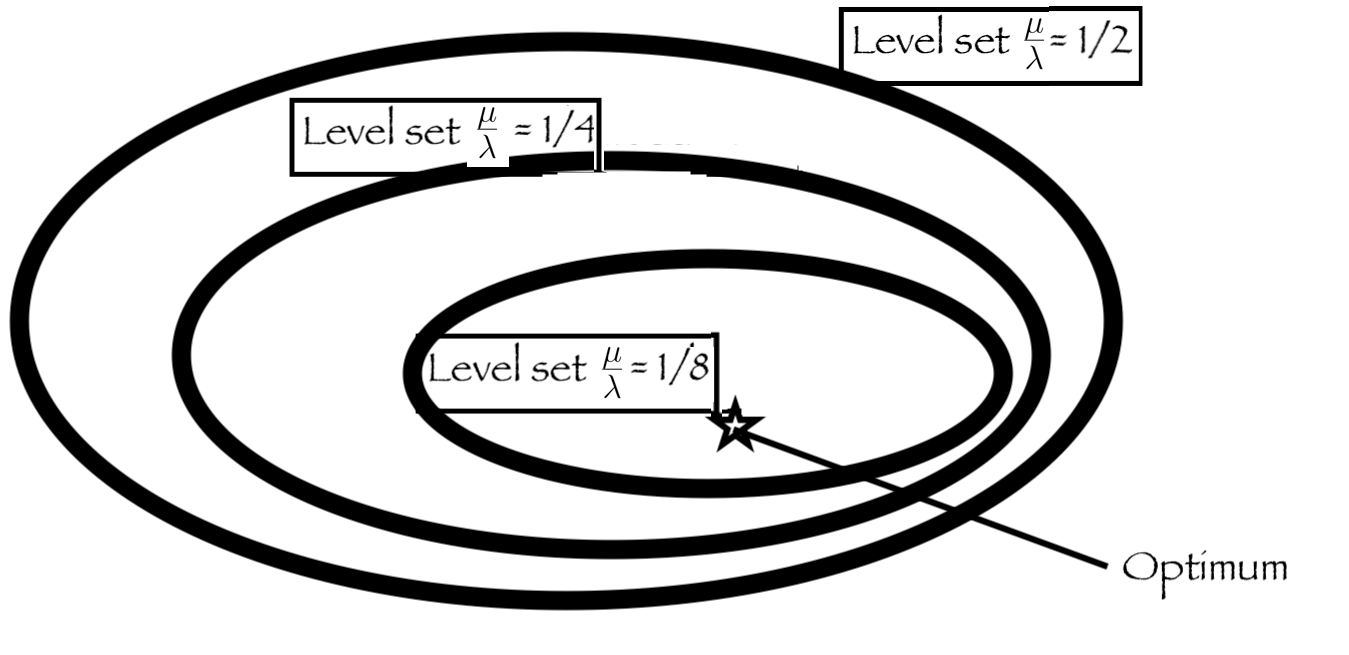}
	\caption{\label{zoli}Assume that we consider a fixed ratio $\mu/\lambda$ and that $\lambda$ goes to $\infty$. The average of selected points, in an unweighted setting and with uniform sampling, converges to the center of the area corresponding to the ratio $\mu/\lambda$: we will not converge to the optimum if that optimum is not the middle of the sublevel. This explains why we need $\mu/\lambda\to 0$ as $\lambda \to \infty$: we do not want to stay at a fixed sublevel. }
	\label{fig:levelset}
\end{figure}
\begin{rmq}[Comparison with random search] The asymptotic
rate obtained for the $\mu$-best averaging approach is of order $\lambda^{-\frac{2(\alpha-1)}{d+2(\alpha-2)}}$,
which is strictly better than the $\lambda^{-2/d}$ rate obtained
with random search, as soon as $d>2$ (because $\alpha>2$) .
This theorem then proves our claim on a wide range of functions. 
\end{rmq}


\begin{rmq}[Comparison with~\cite{ppsnkbest}] ~\cite{ppsnkbest} obtained a rate of order $\lambda^{-1}$ for the sphere function. This rate is better than the one described in Theorem~\ref{thm:principal}. This comes from the bias term in Lemma~\ref{lem:condit-upper}. Indeed for the sphere function, sublevel sets are symmetric, hence the bias term equals $0$, which is not the case in general for functions satisfying Assumption~\ref{ass:principal}. In this paper we are able to deal with potentially non symmetric functions. One can remark, that if the sublevel sets are symmetric the bias term vanishes and we recover the rate of~\cite{ppsnkbest}.
\end{rmq}

\section{Handling wider classes of functions}
\label{sec:wider}
The results we proved are valid for functions satisfying Assumption~\ref{ass:principal}. In particular, the functions are supposed to be regular and have a unique optimum point. In this section, we propose to extend our results to wider classes of functions.

\subsection{Invariance by composition with non-decreasing functions}\label{invar}

{Mathematical results are typically proved under some smoothness assumptions: however, algorithms enjoying some invariance to monotonic transformations of the objective functions do converge on wider spaces of functions as well \cite{monoto}.}
Since the method is based on comparison between the samples, the rank is invariant when the function $f$ is composed with a strictly increasing function $g$. Let $f$ be a function satisfying Assumption~\ref{ass:principal} and $g$ be a strictly increasing function. Consider $h=g\circ f$. Then $h$ admits a unique minimum $x^\star$ coinciding with the one of $f$. As such, the expectation  $\mathbb{E}_{X_{1},...X_{\lambda}\sim U(B(0,r))}\left[\lVert X_{(\mu)}-x^\star\rVert^2\right]$ satisfies the same rates than Theorem~\ref{thm:principal}.
This an immediate consequence of Lemma~\ref{lem:sandwich}. In particular, using the square distance criteria, the rate are preserved even for potentially non regular functions. 
\newotc{For example, our theorem can be adapted to convex piecewise-linear functions, compositions of quadratic functions with non-differentiable increasing functions, and many others.}
Results based on surrogate models  are not applicable here.

\subsection{Beyond unique optima: the convex hull trick, revisited}\label{nonqc}

One of the drawbacks of averaging strategies is that they do not work when there are two basins of optima. For instance, if the two best points $x_{(1)}$ and $x_{(2)}$ have objective values close to those of two distinct optima $x^{\star},y^{\star}$ respectively then averaging $x_{(1)}$ and $x_{(2)}$ may result in a point whose objective value is close to neither. However, in the presence of quasi-convexity this can be countered. It thus makes sense to take into account the possible obstructions to the quasi-convexity of the function and try to counter these, while still maintaining the same basic algorithm as in the case of a unique optimum. \cite{ppsnkbest} proposed to take into account contradictions to quasi-convexity by restricting the number $\mu$ of points used in the averaging. Based on their ideas, we propose the following heuristic.

Let us fix the number of initially selected  points equal to $\mu_{\max}$. Let $x_{(1)},\dots,x_{(\mu_{\max})}$ be these points ranked from best to worst. Define $S_i=(x_{(1)},\dots,x_{(i)})$ and $C_i$ the interior of the convex hull of $S_i$. Assume that there is no tie in fitness values, that is no $i\neq j$ such that $f(x_{i})=f(x_{j})$. Given $\mu_{\max}$, choose $\mu$ maximal such that
\begin{equation}\forall i\leq \mu, x_{(i)} \not \in C_i.\label{oldeq}\end{equation}

One can remark that
$x_{(\mu)}\in C_\mu\Rightarrow f\mbox{ is not quasi-convex on }C_\mu$. However, this may not detect all cases in which $f$ is not quasi-convex on $C_\mu$. More generally,
\begin{equation}\exists j > \mu-1,\ x_{(j)}\in C_\mu \Rightarrow f \mbox{ is not quasi-convex on }C_\mu.\label{neweq}\end{equation} 
If such a $j$ is not $\mu$, Eq. \eqref{oldeq} does not detect the non-quasiconvexity: therefore, \eqref{neweq} detects more non-quasiconvexities than Eq. \eqref{oldeq}.

Therefore we choose $\mu$ maximal such that for all $i<\mu, j>i$, $x_{(j)}\not\in C_i$. This heuristic leads to a choice of average which is "consistent" with the existence of multiple basins.

\begin{figure*}[!h]
    \centering
    \includegraphics[width=0.3\textwidth]{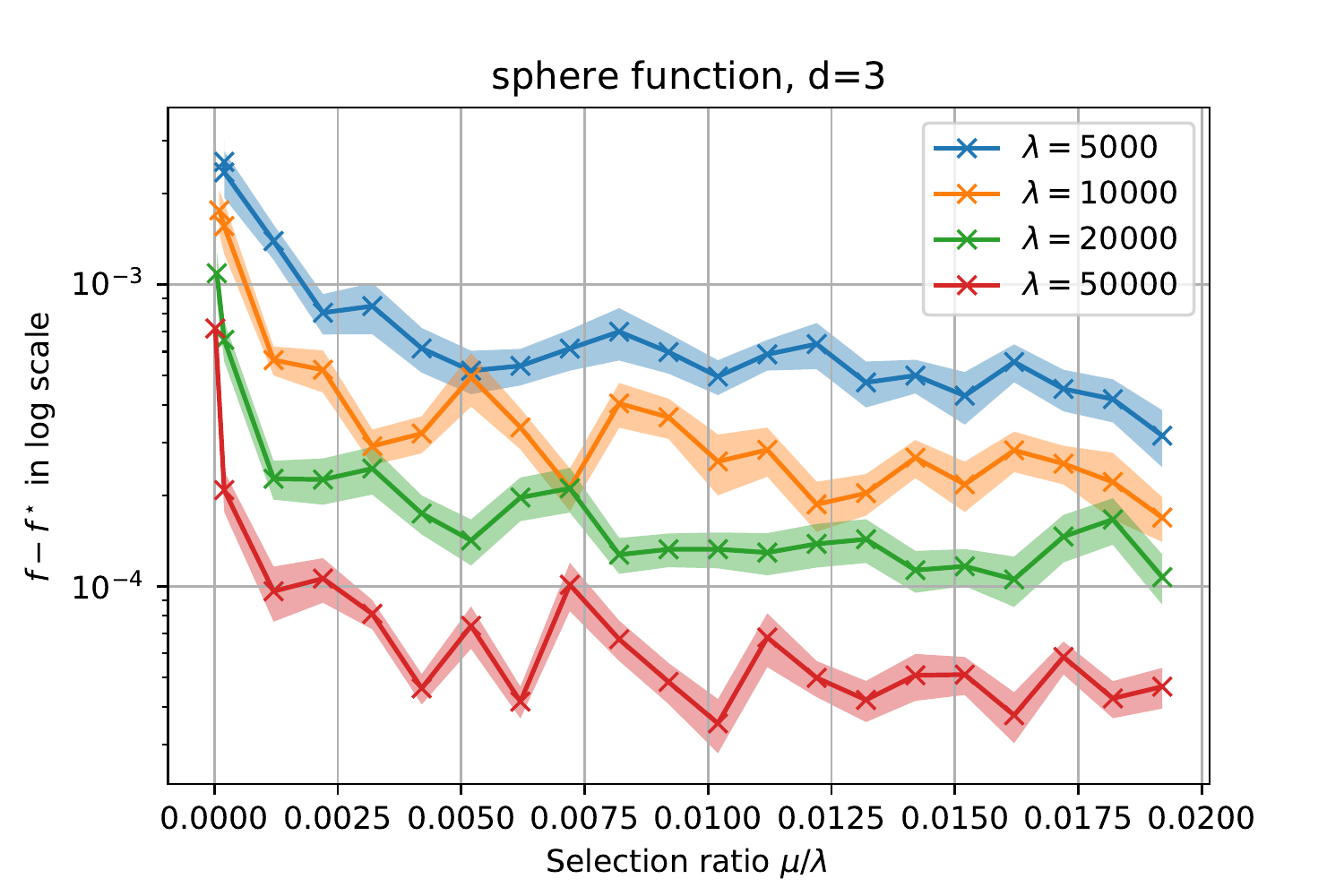}~ \includegraphics[width=0.3\textwidth]{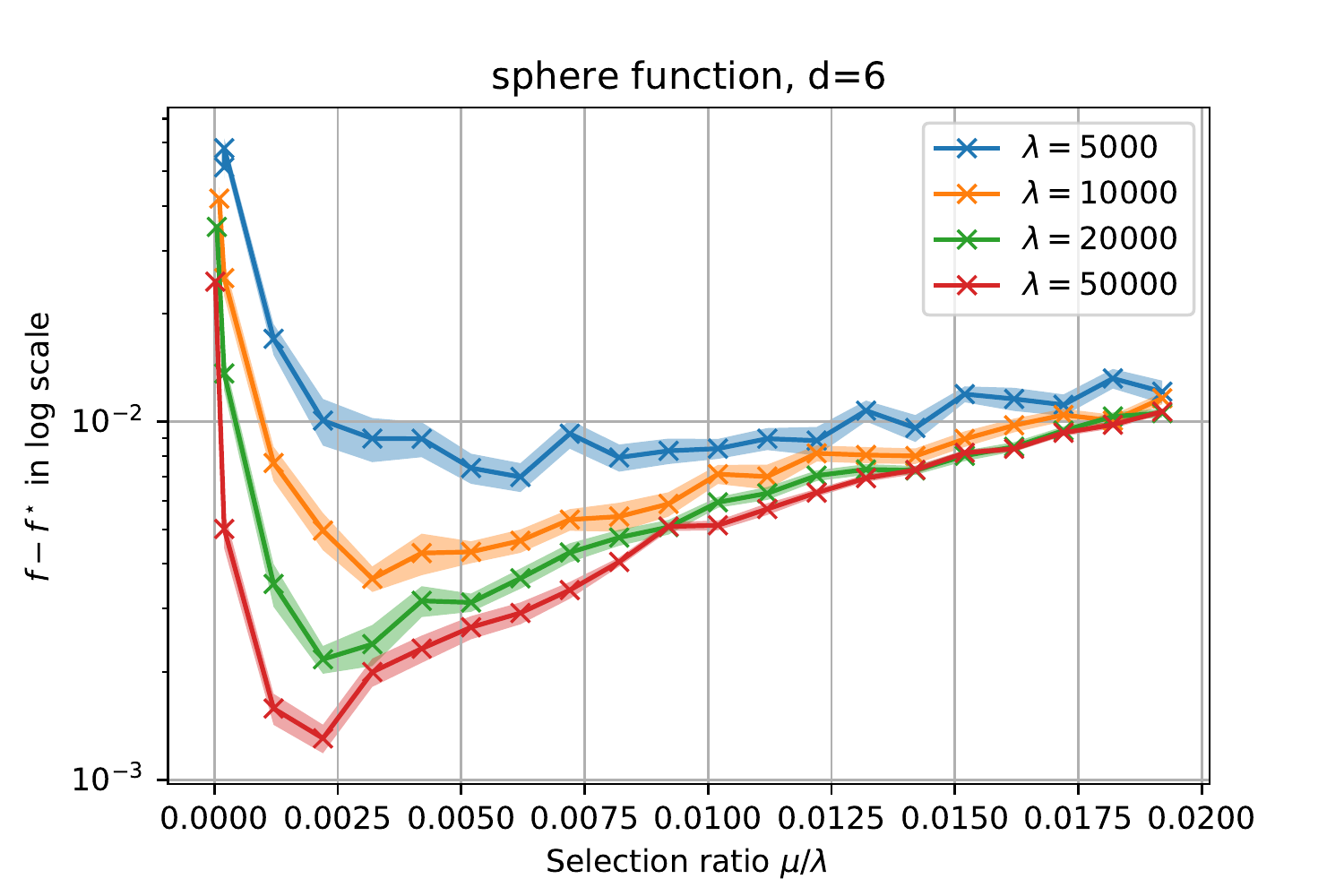}~    \includegraphics[width=0.3\textwidth]{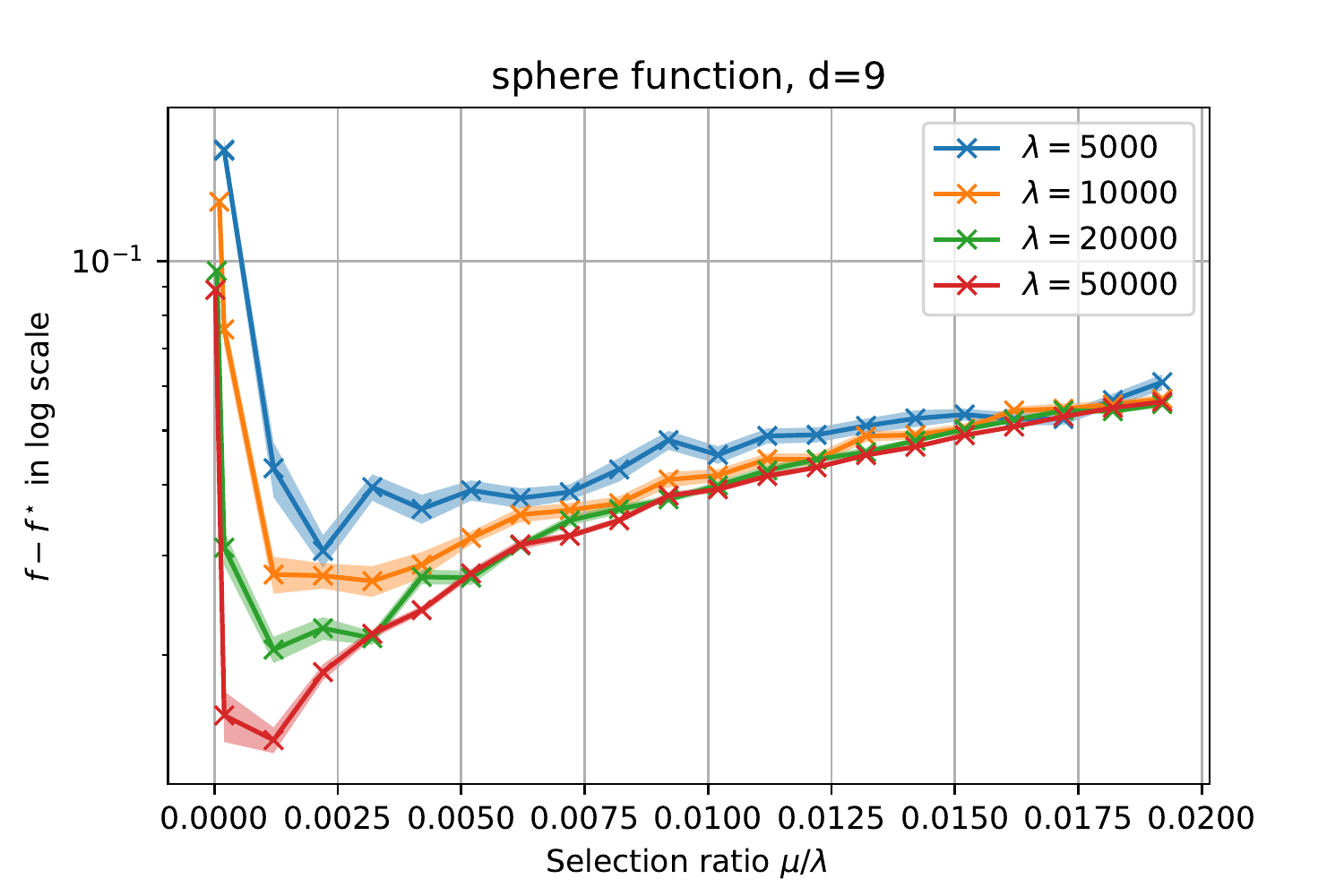}\\ Sphere function\\
    \includegraphics[width=0.3\textwidth]{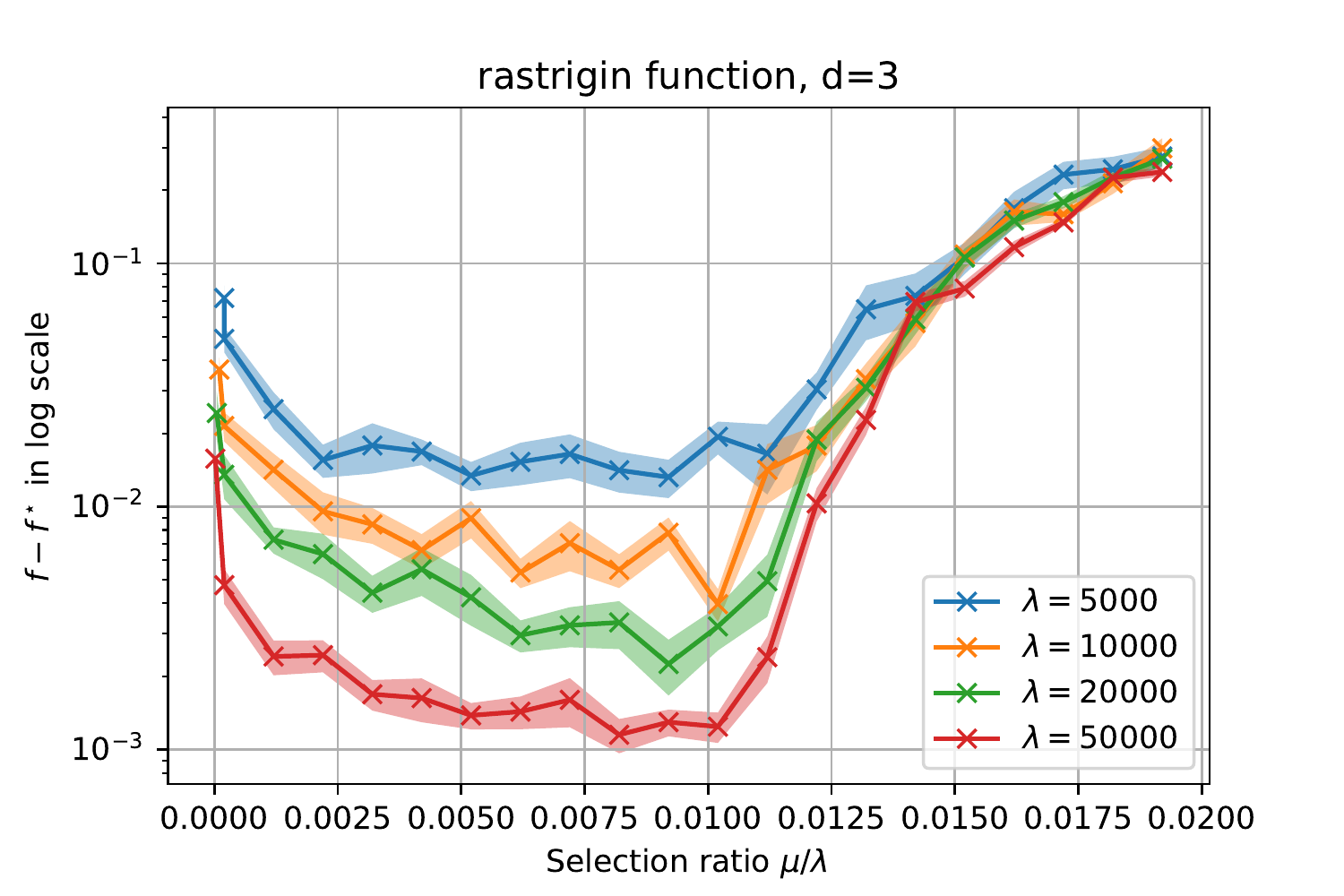}~ \includegraphics[width=0.3\textwidth]{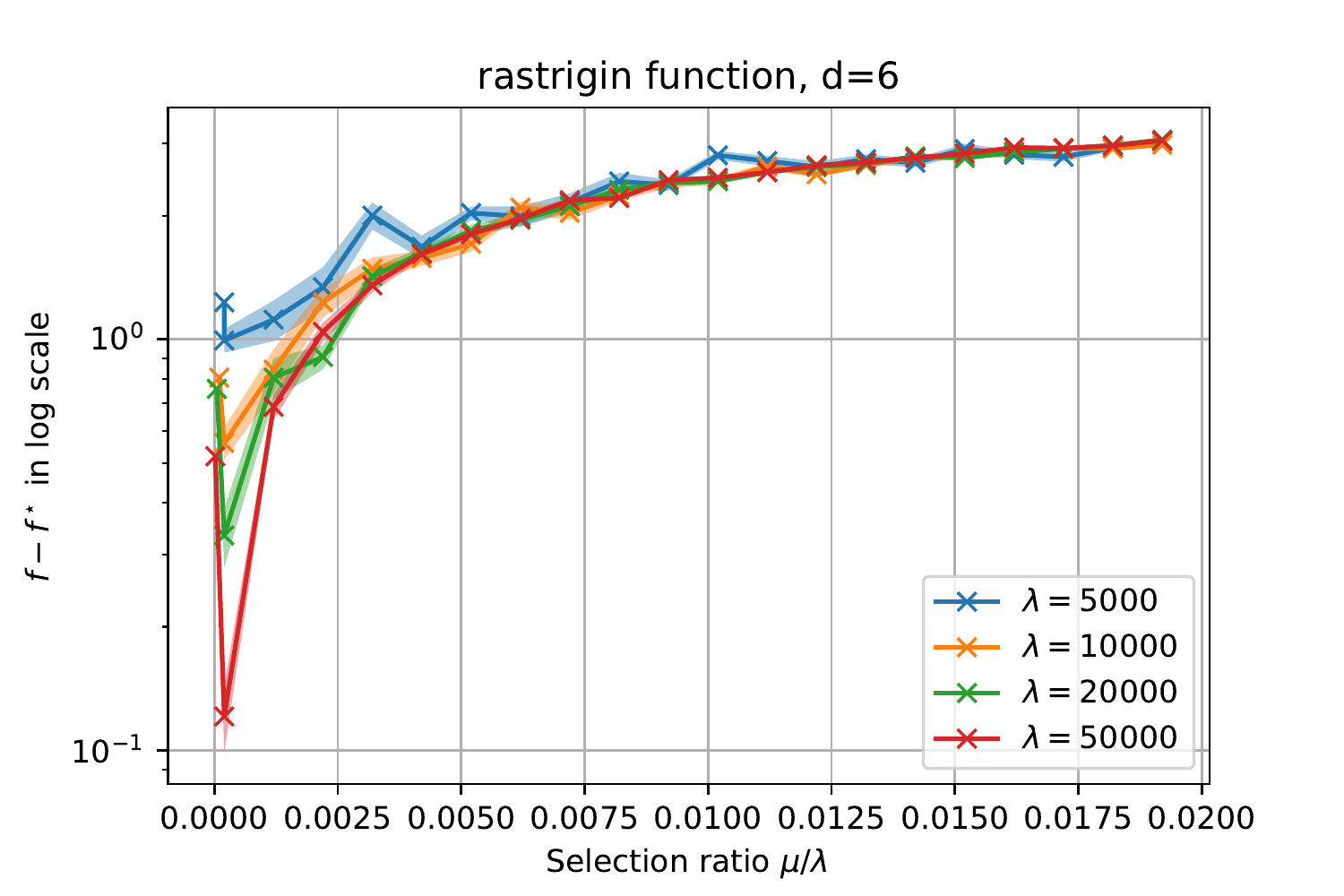}~     \includegraphics[width=0.3\textwidth]{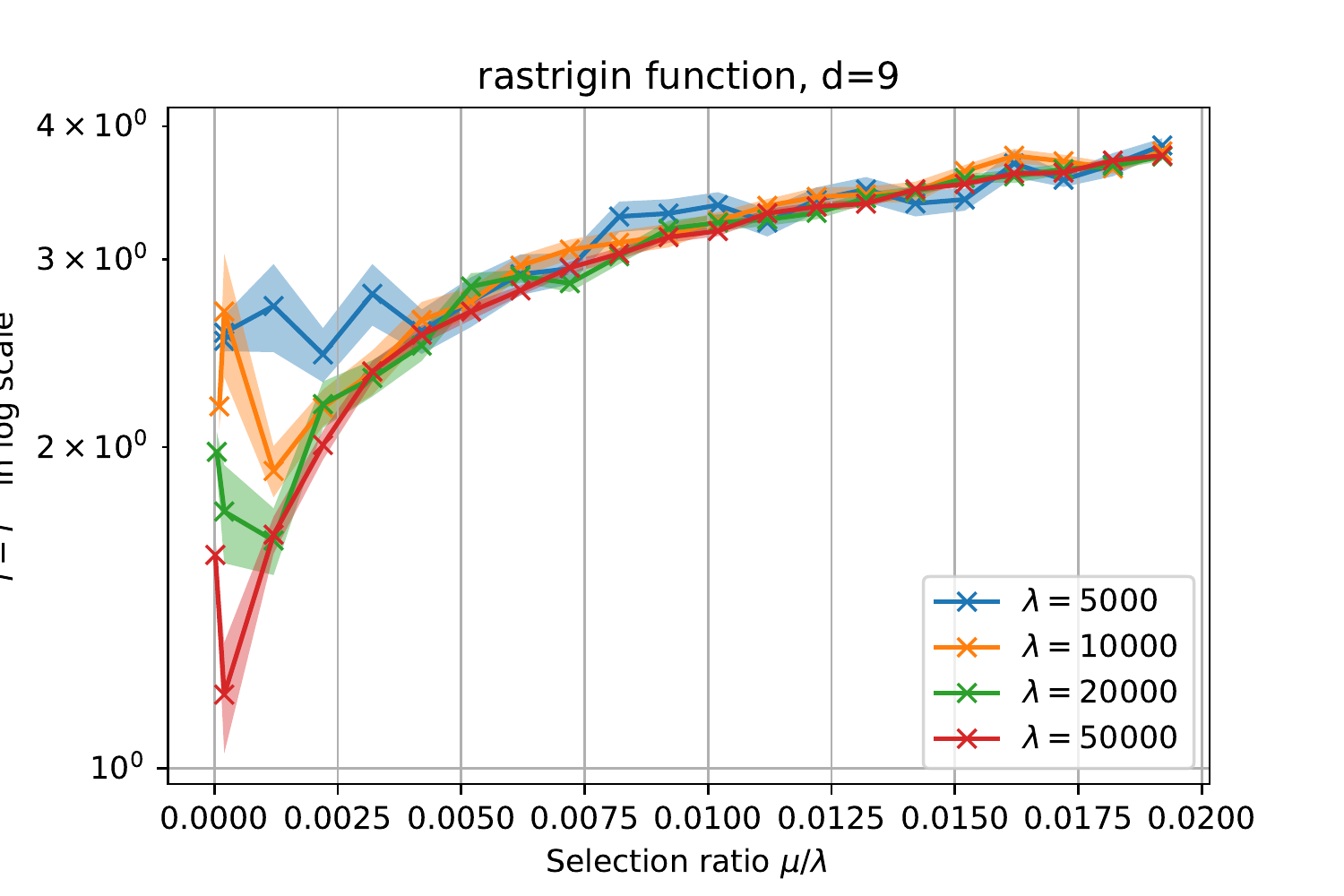}\\ Rastrigin function\\
    
        \includegraphics[width=0.3\textwidth]{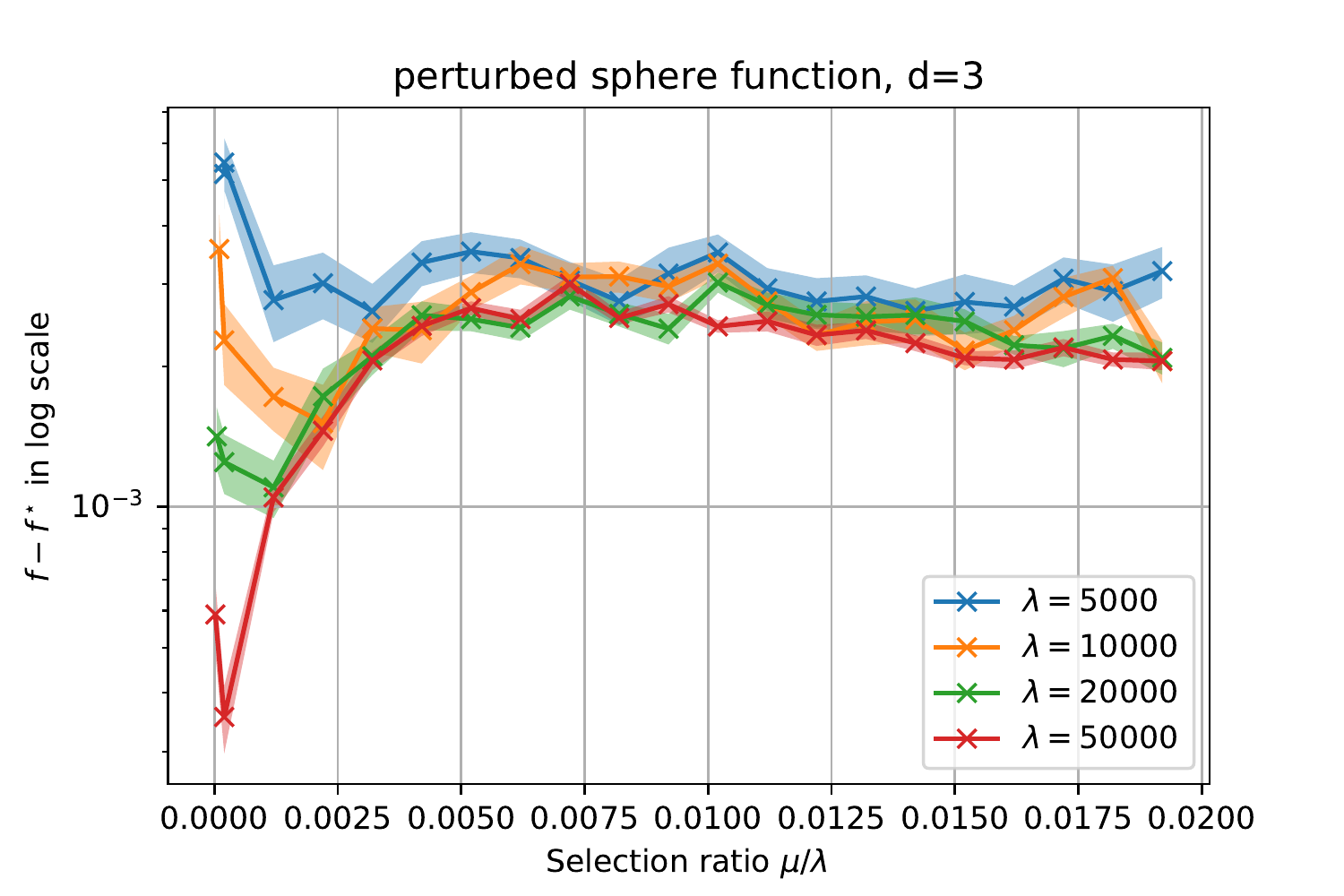}~  \includegraphics[width=0.3\textwidth]{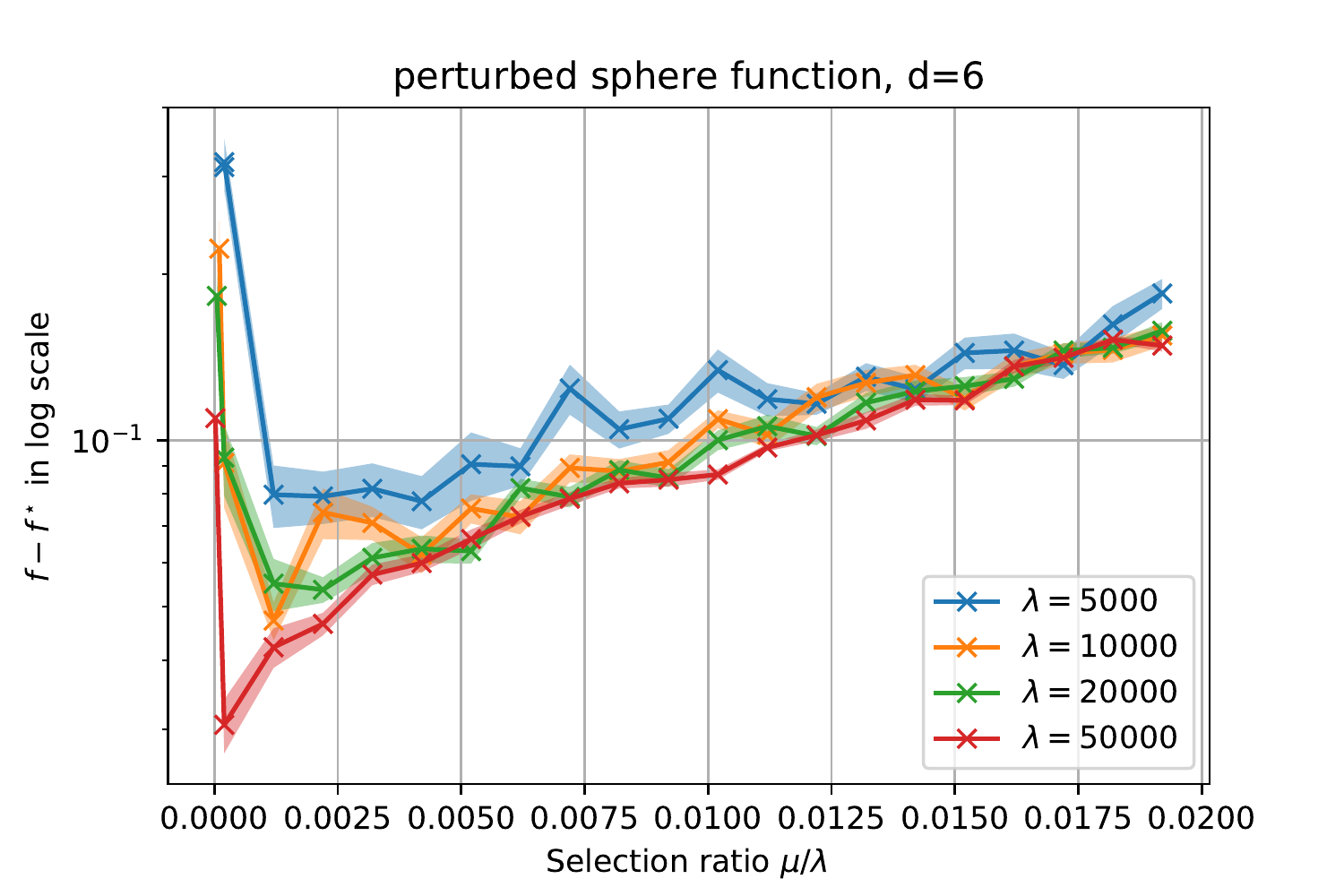}~     \includegraphics[width=0.3\textwidth]{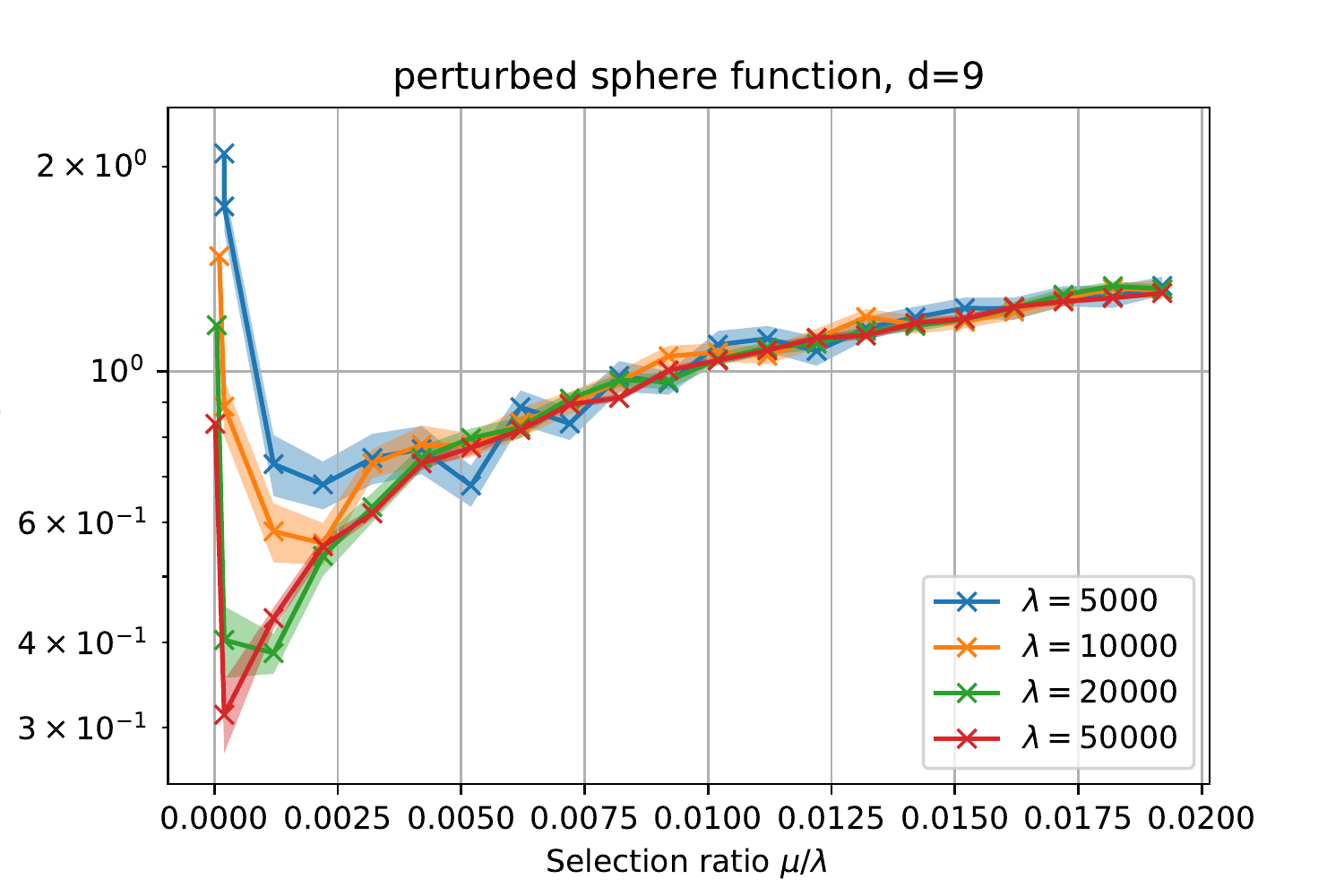}\\ Perturbed sphere function\\
    \caption{Average regret $f(\bar{X}_{(\mu)})-f(x^\star)$ in logarithmic scale in function of the selection ratio $\mu /\lambda$ for different values of $\lambda\in\{5000,10000,20000,50000\}$. The experiments are run on Sphere, Rastrigin and Perturbed Sphere function for different dimensions  $d\in \{3,6,9\}$. All results are averaged over $30$ independent runs. We observe, consistently with our theoretical results and intuition, that (i) the optimal $r=\frac{\mu}{\lambda}$ decreases as $d$ increases (ii) we need a smaller $r$ when the function is multimodal (Rastrigin) (iii) we need a smaller $r$ in case of dissymmetry at the optimum (perturbed sphere).}
    \label{fig:examples}
\end{figure*}
\section{Experiments}
\label{sec:xps}
We divide the experimental section in two parts. In a first part, we focus on validating theoretical findings, then we compare with existing optimization methods.

\subsection{Validation of theoretical findings}

In this section, we will assume that $r=1$ and that the optimum $x^*$ will be sampled uniformly in the ball of radius $0.9$. We compare results on the following functions:
\begin{enumerate}
    \item Sphere function: 
    \begin{align*}
        f(x)=\sum_{i=1}^d (x_i-x_i^\star)^2
    \end{align*}
    \item Rastrigin function: 
    \begin{align*}
        f(x)=\sum_{i=1}^d (x_i-x_i^\star)^2 + 1-\cos{\left(2\pi (x_i-x_i^\star) \right)}
    \end{align*}
    \item Perturbed sphere function:
       \begin{align*}
        f(x)=\sum_{i=1}^d (x_i-x_i^\star)^2 +\left(\sum_{i=1}^d g(x_i-x_i^\star) \right)^3
    \end{align*}
    with $g(x) = x$ if $x>0$ and $-2x$ otherwise. This function has highly non symmetric sublevel sets, but still satisfies Assumption~\ref{ass:principal}.
\end{enumerate}

We plotted in Figure~\ref{fig:examples} the regret $f(\bar{X}_{(\mu)})-f(x^\star)$ as a function of $\mu/\lambda$ for different dimensions $d$ and number of samples $\lambda$. The experiments are averaged over $30$ runs. We remark for instance on the Rastrigin function that for the $\mu$-best averaging approach to be better than random search, we need a very large number of samples as the dimension increases. Overall, these plots validate our theoretical findings that averaging a few best points leads to a better regret than only taking the best one.

\subsection{Comparison with other methods}
\begin{figure*}
    \centering
    \includegraphics{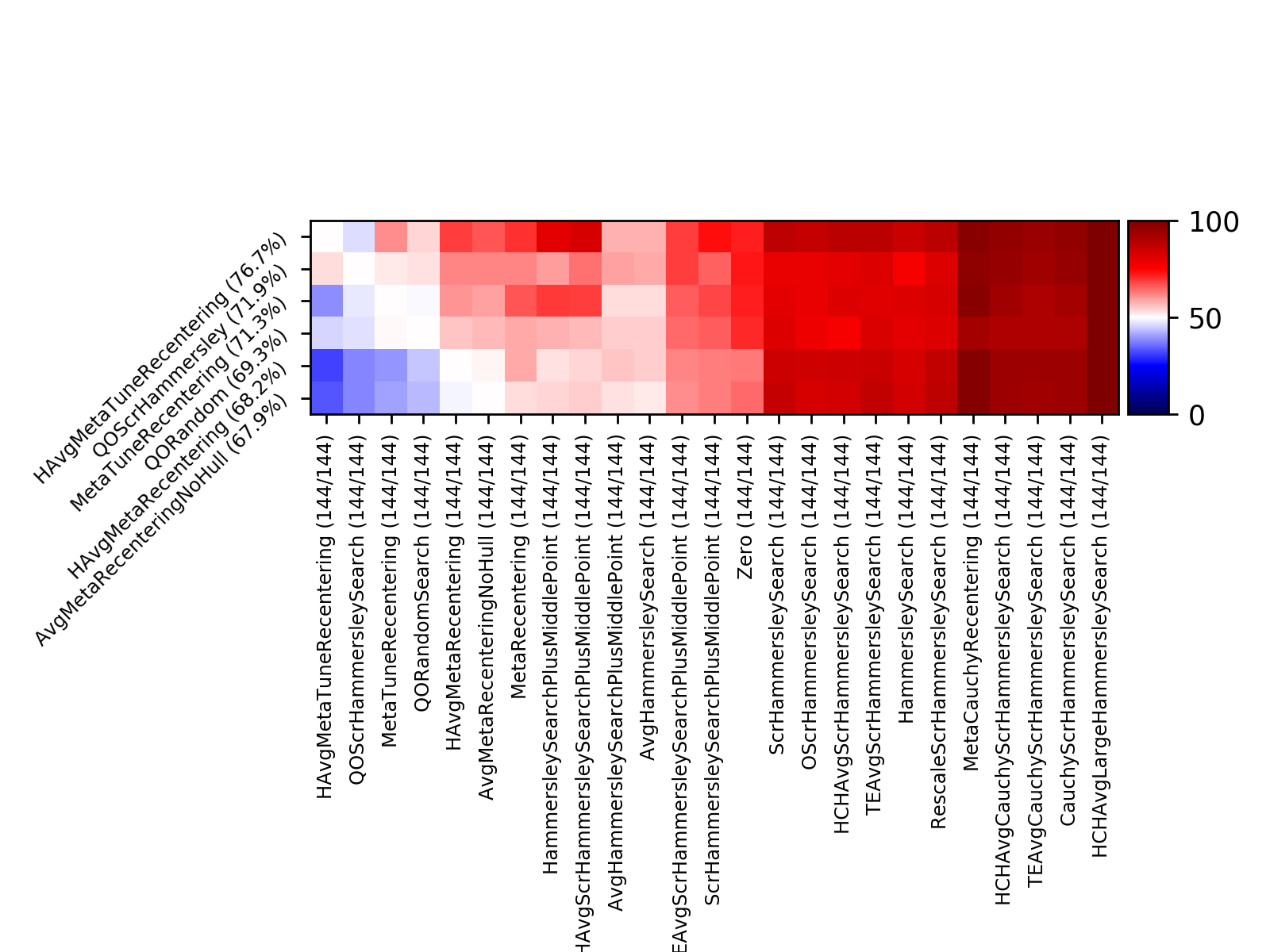}
    \caption{Experimental results: row A and col B presents the frequency (over all 144 test cases) at which A outperforms B in terms of average loss. Then rows are sorted per average winning rate and we keep the 6 best ones. Zero is a naive method just choosing zero: we see that, consistently with \cite{icmldoe}, many methods are worse than that when the dimension is huge compared to the budget.}
    \label{figxp}
\end{figure*}
In this section, we compare averaging strategies with other standard strategies, using the Nevergrad library~\cite{nevergrad}. 
Figure \ref{figxp} presents experimental results based on Nevergrad. Instead of the uniform sampling used in the theoretical results and the previous experimental validation, we use Gaussian sampling in this set of experiments. Following the notation from \cite{ppsnkbest}, we consider distinct averaging prefixes:
\begin{itemize}
    \item \texttt{AvgXX} = method \texttt{XX}, plus averaging of the $\mu=\lambda/(1.1^d)$ best points in dimension $d$.
    \item \texttt{HAvgXX} = = method  \texttt{XX}, plus averaging of the $\lambda/(1.1^d)$ best points, restricted by \newotc{the convex hull trick (Section \ref{nonqc}).}
\end{itemize}
\newotc{Many other methods are included: we refer to \cite{nevergrad} for more information.}
Recently, \cite{icmldoe,ppsnrescaling} pointed out that when the optimum is randomly drawn from a standard normal distribution, we should use rescaling methods for focusing closer to the center in high dimensional setting. Several such methods have been proposed:
\begin{itemize}
\item \texttt{QOXX} = method \texttt{XX}, plus quasi-opposite sampling \cite{quasiopposite}, i.e. each time we draw $x$ with ${\mathcal{N}}$, we also use $-rx$ where $r$ is uniformly independently drawn in $(0,1)$.
\item \texttt{XXPlusMiddlePoint} = method \texttt{XX}, except that there is one point forced at the center of the domain. 
\item \texttt{MetaRecentering} \cite{icmldoe}:  rescaling $\sigma = (1+\log(n))/(4\log(d))$, i.e. we randomly draw with $\sigma \times {\mathcal{N}(0,I_d)}$ instead of ${\mathcal{N}(0,I_d)}$.
\item \texttt{MetaTuneRecentering} \cite{ppsnrescaling}: rescaling $\sigma = \sqrt{\log(\lambda) / d}$, i.e. we randomly draw with $\sigma \times {\mathcal{N}(0,I_d)}$ instead of ${\mathcal{N}(0,I_d)}$.
\end{itemize}

\paragraph{Experimental setup.}

We measure the simple regret and compare methods by average frequency of win against other methods. For each test case, we randomly draw the optimum as ${\mathcal {N}(0,I_d)}$ (multivariate standard Gaussian), with
different budgets $\lambda$ in $\{30, 100, 300,$ $ 1000,3000, 10000, 100000\}$ and
dimensions $d$ in $\{3,10,30,100,300,$ $1000, 3000\}$. Due to their time of evaluation, we did not run the cases with both $d=3000$ and $\lambda = 100000$. We evaluated on 3 different functions: the sphere function, the Griewank function, and the Highly Multimodal function. Previous results~\cite{bousquet} from the literature have already shown that replacing random sampling by scrambled Hammersley sampling (i.e. modern low discrepancy sequences compatible with high dimension) leads to better results.

\paragraph{Analysis of results.}

Analyzing the table results from Figure \ref{figxp}, we observe that
\begin{itemize}
\item Averaging performs well overall: \texttt{AvgXX} is better than \texttt{XX};
\item The quasi-convex trick from Section \ref{nonqc} does work: \texttt{HAvgXX} is better than \texttt{AvgXX};
\item The rescaling strategy from \cite{ppsnrescaling} outperforms the ones in \cite{icmldoe} 
 (\texttt{MetaTuneRecentering} better than \texttt{MetaRecentering} or than \texttt{PlusMiddlePoint}) which are already better than standard quasi-random sampling. Quasi-Opposite sampling is also competitive.
 \end{itemize}
 We also include various methods present in the platform, including those which are based on Cauchy or Hammersley without scrambling (Hammersley in the name without ``Scr'' prefix), or sophisticated uses of convex hulls for estimating the location of the optimum (HCH in the name).

\section{Conclusion}
We proved that averaging $\mu>1$ points rather than picking up the best works even for non quadratic functions, in the sense that the convergence rate is better than the one obtained just by picking up the best point. We also proved faster rates than methods based on meta-models (such as \cite{bach}) unless the objective function is very smooth and low dimensional. \newotc{We also show that our results cover a wider family of functions (Section \ref{invar}).}
\newotc{We also propose a rule for choosing $\mu$, depending on $\lambda$ and the dimension. This shows that the optimal $\mu/\lambda$ ratio decreases to $0$ as the dimension goes to infinity, which is confirmed by Fig. \ref{fig:examples}. We also note, by comparing with \cite{ppsnkbest}, that the optimal ratio should be smaller (Fig. \ref{zoli}), which is confirmed by our experiments on the perturbed sphere (Fig. \ref{fig:examples}).
We also propose a method for adapting this $\mu$, by automatically detecting non-quasi-convexity and reducing it: and prove that it detects more non-quasiconvexities than the method proposed in \cite{ppsnkbest}. Finally, we validate the approach on a reproducible open-sourced platform (Fig. \ref{figxp}).}

\subsection*{Further work}
Using density-dependent weights as in \cite{sumo} should allow us to get rid of the constraint $||x^*||<r$ using a Gaussian sampling instead of a uniform sampling. Better rates might be obtained with rank-dependent weights as in \cite{arnoldweights}. We also leave as further work the proof of the optimality of the rate for this strategy. Moreover, we also believe better rates can be obtained for smoother functions, and leave this study for further work.
\newotc{The case of noisy objective functions \cite{arnoldbeyer} is critical. The study is harder, and good evolutionary algorithms use large populations, making the overall algorithm closer to a small number of one-shot optimization algorithms: actually, some fast algorithms use mainly learning \cite{astete,clop,noisymesh}. Population control\cite{beyerhellwignoise} is successful and its last stage looks exactly like a one-shot optimization method.}
\FloatBarrier
\balance
\bibliographystyle{ACM-Reference-Format}
\bibliography{samples/biblio}
\end{document}